\documentclass[12pt, aps, showpacs, pra, superscriptadaddress, eqsecnum,notitlepage,USLetter,nofootinbib]{revtex4-1}
\usepackage[utf8]{inputenc}
\usepackage{titlesec}
\usepackage{hyperref}
\usepackage{amssymb}
\usepackage{amsthm}
\usepackage{ulem}
\usepackage{amsmath}
\usepackage{amsfonts}
\usepackage{epstopdf}
\usepackage{mathtools,nccmath,relsize}
\usepackage[dvipsnames]{xcolor}
\usepackage{bbm}
\usepackage{enumerate}
\usepackage{calrsfs}
\usepackage{bigints}
\usepackage{autonum}
\usepackage[makeroom]{cancel} 

\newtheorem{Th}{Theorem}[section]
\newtheorem{Prop}[Th]{Proposition}
\newtheorem{Lem}[Th]{Lemma}

\theoremstyle{Remark}
\newtheorem{Rmk}[Th]{Remark}
\theoremstyle{Definition}
\newtheorem{Def}[Th]{Definition}
\theoremstyle{Corollary}
\newtheorem{Cor}[Th]{Corollary}

\theoremstyle{plain} 
\newcommand{\thisTheoremname}{}
\newtheorem{genericthm}[Th]{\thisTheoremname}

\allowdisplaybreaks

\newcommand{\RRR}{\mathbb{R}^3}
\newcommand{\R}{\mathbb{R}}
\newcommand{\Hone}{H^1(\mathbb{R}^3)}
\newcommand{\BV}{BV(\mathbb{R}^3)}
\newcommand{\Lone}{L^1(\mathbb{R}^3)}
\newcommand{\Loneloc}{L^1_{loc}(\mathbb{R}^3)}
\newcommand{\Lthreehalves}{L^\frac32(\mathbb{R}^3)}
\newcommand{\Lthreehalvesloc}{L^\frac32_{loc}(\mathbb{R}^3)}
\newcommand{\Ltwo}{L^2(\mathbb{R}^3)}
\newcommand{\Ltwoloc}{L_{loc}^2(\mathbb{R}^3)}
\newcommand{\Ltenthirds}{L^\frac{10}{3}(\mathbb{R}^3)}

\newcommand{\Lfourthirds}{L^\frac43 (\mathbb{R}^3)}
\newcommand{\Lfivehalves}{L^\frac52 (\mathbb{R}^3)}

\newcommand{\Linfty}{L^\infty(\mathbb{R}^3)}
\newcommand{\eps}{\varepsilon}
\newcommand{\Om}{\Omega}

\newcommand{\NN}{\mathbb{N}}
\newcommand{\Eeps}{\mathcal{E}_\eps}
\newcommand{\EepsV}{\mathcal{E}_\eps^V}
\newcommand{\EepsnV}{\mathcal{E}_{\eps_n}^V}
\newcommand{\EepsZ}{\mathcal{E}_\eps^Z}
\newcommand{\Ekz}{\mathcal{E}_\ek^0}
\newcommand{\EepsnZ}{\mathcal{E}_{\eps_n}^Z}
\newcommand{\EkV}{\mathcal{E}_\ek^V}
\newcommand{\EzV}{\mathcal{E}_0^V}
\newcommand{\Ez}{\mathcal{E}_0}
\newcommand{\EzZ}{\mathcal{E}_0^Z}
\newcommand{\Ezz}{\mathcal{E}_0^0}
\newcommand{\ezz}{e_0^0}
\newcommand{\ezV}{e_0^V}
\newcommand{\ezZ}{e_0^Z}
\newcommand{\eepsV}{e_\eps^V}
\newcommand{\eepsnV}{e_{\eps_n}^V}
\newcommand{\eepsnZ}{e_{\eps_n}^Z}
\newcommand{\HM}{\mathcal{H}^M}
\newcommand{\Hz}{\mathcal{H}_0^M}
\newcommand{\Hzp}{\mathcal{H}_{0,+}^M}
\newcommand{\BVStep}{BV(\RRR,\{0,\pm1\})}
\newcommand{\XX}{\BVStep}
\newcommand{\Xp}{\mathcal{X}_+^M}
\newcommand{\XM}{\mathcal{X}^M}
\newcommand{\Fz}{\mathcal{F}_0}
\newcommand{\FzV}{\Fz^V}
\newcommand{\FzZ}{\Fz^Z}

\newcommand{\BRM}{{\mathbb{B}_M}}
\newcommand{\MM}{\mathcal{M}}
\newcommand{\epk}{{\varepsilon_k}}
\newcommand{\ek}{{\epk}}
\newcommand{\rhk}{{\rho_k}}
\newcommand{\rhko}{{\rho^1_k}}
\newcommand{\One}{{\mathbbm{1}}}
\newcommand{\Ui}{\{u^i\}_{i=0}^\infty}
\newcommand{\supp}{\text{supp}\,{}}

\newcommand{\ktends}{\xrightarrow[k \rightarrow \infty]{}}
\newcommand{\ntends}{\xrightarrow[n \rightarrow \infty]{}}
\newcommand{\etends}{\xrightarrow[\eps \rightarrow 0^+]{}}

\newcommand{\beq}{\begin{equation}}
\newcommand{\eeq}{\end{equation}}
\newcommand{\bal}{\begin{align}}
\newcommand{\eal}{\end{align}}

\begin{document}


\title{Convergence of the TFDW Energy to the liquid drop Model}

\author{Lorena Aguirre Salazar}
\affiliation{Department of Mathematics and Statistics, McMaster University, 1280 Main Street West, Hamilton, Ontario, Canada}
\author{Stan Alama}
\affiliation{Department of Mathematics and Statistics, McMaster University, 1280 Main Street West, Hamilton, Ontario, Canada}
\author{Lia Bronsard}
\affiliation{Department of Mathematics and Statistics, McMaster University, 1280 Main Street West, Hamilton, Ontario, Canada}

\begin{abstract}

\begin{center}\textbf{Abstract}\end{center}
\baselineskip=15pt
We consider two nonlocal variational models arising in physical contexts.  The first is the Thomas-Fermi-Dirac-von Weizs\"{a}cker (TFDW) model, introduced in the study of ionization of atoms and molecules, and the second is the liquid drop model with external potential, proposed by Gamow in the context of nuclear structure.  It has been observed that the two models exhibit many of the same properties, especially in regard to the existence and nonexistence of minimizers.  We show that, under a ``sharp interface'' scaling of the coefficients, the TFDW energy with constrained mass $\Gamma$-converges to  the  liquid drop model, for a general class of external potentials.  Finally, we present some consequences for global minimization of each model.
\end{abstract}

\maketitle

\baselineskip=16pt

\section{Introduction}

The Thomas-Fermi-Dirac-von  Weizs\"{a}cker (TFDW) theory is a variational model for ionization in atoms and molecules.  Minimizers $u\in H^1(\RRR)$ of the energy
\begin{align}\label{TFDW}
 E_{TFDW}(u) &= \int_{\RRR} \left(  c_{TF} |u|^{\frac{10}{3}} - c_{D} |u|^{\frac83 }
        + c_W |\nabla u|^2 - V|u|^2 \right) dx+D(\vert u\vert ^2,\vert u\vert^2)
\end{align}
where
\begin{align}D(f,g):=\frac12 \int_{\RRR}\int_{\RRR} {f(x)g(y)\over |x-y|} dx\, dy,\end{align}
subject to an $L^2$ constraint, $\|u\|_{\Ltwo}^2=M$, model electron density in an atom or molecule whose nuclei act via the electrostatic potential $V$, and total electron charge $M$   (see \cite{Lieb}.) The liquid drop model (with potential) is also a variational problem with physical motivations: for sets $\Om\subset\RRR$ of finite perimeter and given volume $|\Om|=M$, one minimizes  the energy
$$   E_{LD}(\Om) = \text{Per}_{\RRR}(\Om) - \int_{\Om} Vdx+ D(\mathbbm{1}_\Omega,\mathbbm{1}_\Omega) .  $$
Here, the first term represents the perimeter of $\partial\Om$, which may be calculated as the total variation of the measure $|\nabla \mathbbm{1}_\Om|$, with $\mathbbm{1}_\Om\in  BV(\RRR; \{0,1\})$.  When $V\equiv 0$, this is Gamow's problem, a simplified model for the stability of atomic nuclei   (see \cite{ChoksiMuratovTopaloglu}) .  The constraint value $M$ represents the number of nucleons bound by the strong nuclear force.  

As variational problems, the TFDW and liquid drop models have much in common.   Each features a competition between local attractive terms (gradient and potential terms) and a common non-local repulsive term.  As such, each problem is characterized by subtle problems of existence and nonexistence due to the translation invariance of the problem ``at infinity'':  for large values of the ``mass'' constraint $M$, minimizing sequences may fail to converge due to splitting of mass which escapes to infinity, the ``dichotomy'' case in the concentration-compactness principle of Lions~\cite{LionsConcentrationPart1}.  
 (See e.g., \cite{BonaciniCristoferi,CP,FigalliFuscoMaggiMillotMorini,FrankLieb,FrankNamVBosch,KnupferMuratov,KMN,Lions,LuOttoNon-existence,LuOttoLiquidDropBackground,NamVBosch}.)
While this similarity has been often remarked, and one often speaks of the liquid drop models as a sort of ``sharp interface'' version of TFDW, no direct analytic connection between the two has been made.  In this paper we prove that, after an appropriate  ``sharp interface'' scaling and normalization, the TDFW energy converges to the liquid drop model with potential,  within the context of $\Gamma$-convergence.  
  This result may seem a bit surprising, since in {\it bounded} domains $\Om\subset\RRR$ it is the Ohta-Kawasaki functional, arising in di-block copolymer models and with an $L^1(\Om)$ mass constraint, which $\Gamma$-converges to the nonlocal isoperimetric problem (which is a bounded domain form of the liquid drop model); see \cite{RenTruskinovsky,RenWei200,Muratov,ChoksiMuratovTopaloglu}.

In order to establish this connection we select the constants in the TFDW energy so as to set up a sharp interface limit.  We note that this choice of scaling is not physically natural for the application to ionization phenomena, but is motivated purely mathematically.  We introduce a length-scale parameter $\eps>0$, and choose constants $c_W={\eps\over 2}$, $c_{TF}=\frac{1}{2\eps}$ and $c_D=\frac{1}{\eps}$.  We note that for fixed $\eps$, the qualitative behavior of the minimization problem for TFDW is not affected by the specific choices of the constants $c_W,c_{TF}, c_D$, and the values we select here match the standard choice of constants in the liquid drop model.  In addition, we complete the square in the nonlinear potential by adding in a multiple of the constrained $L^2$ norm, which is a constant in the minimization problem and thus has no effect on the existence of minimizers or the Euler-Lagrange equations.  That is, the nonlinear potential is rewritten as,
$$   \int_{\RRR} \frac{1}{2\eps}\left( |u|^{\frac{10}{3}} - 2|u|^{\frac83}\right) dx
     =  \int_{\RRR} \frac{1}{2\eps} |u|^2 \left( |u|^{\frac23 }-1\right)^2 dx - {M\over 2\eps},  $$
where $M=\|u\|_{\Ltwo}^2$ according to the constraint.  Thus we recognize the triple well potential,
$$   W(u)  : = |u|^2 \left( |u|^{\frac23 }-1\right)^2, $$
vanishing at $|u|=0,1$, and a version of the TFDW energy of the rescaled and normalized form,
\begin{align}\label{EpsFnctnl}
\EepsV(u):=
\int_{\RRR}\left[\frac{\eps}{2} | \nabla u | ^2+ \frac{1}{2\eps}W(u)- V|u|^2\right]d {x}+ D(\vert u\vert ^2,\vert u\vert^2) , \quad \|u\|_{\Ltwo}^2=M.
\end{align}
As $\eps\to 0^+ $ we expect that sequences $ \{u_\eps\}_{\eps>0}$ of uniformly bounded energy should converge almost everywhere to one of the wells of the potential $W$, that is, in the limit $u(x)\in \{0,\pm 1\}$.  As $\EepsV(|u|)=\EepsV(u)$, we expect minimizers of $\EepsV$ to have fixed sign, but families $ \{u_\eps\}_{\eps>0}$ with bounded energy might well take both positive and negative values.  Hence, we define the limiting liquid drop functional for $u\in BV(\RRR; \{0, \pm 1\})$ as
\begin{align}\label{LDen}
\Ez^V(u)&:=\frac{1}{8}\int_{\RRR}|\nabla u| -\int_{\RRR} V\vert u\vert^2dx+ D(\vert u\vert ^2,\vert u\vert^2) .\end{align}
The first term is the total variation of the measure $|\nabla u|$, and for $u=\mathbbm{1}_{\Om}$ it measures the perimeter of $\partial\Om$.
If $u$ takes both values $\pm 1$, then 
$$  \int_{\RRR}  |\nabla u| = \int_{\RRR} |\nabla u_+| + |\nabla u_-| ,$$
which measures the perimeter of $\{x\in\RRR  \ | \ u(x)=1\}$ and that of $\{x\in\RRR  \ | \ u(x)=-1\}$, whereas the other terms yield the same value for $u$ and $\vert u\vert=u^2$.  

We make the following general hypotheses regarding the potential V:
\begin{align}\label{HypV}V\in L^{\frac{5}{2}}(\RRR)+ \Linfty \text{ and } V(x) \xrightarrow[| x |\rightarrow \infty]{}0 .\end{align}
We define domains for the functionals which incorporate the mass constraint,
\begin{gather}
    \HM:=\left\{u\in\Hone\, : \, \| u\|_{\Ltwo}^2=M\right\}, \\
    \XM:=\left\{u\in \XX \, : \,  \| u\|_{\Ltwo}^2=M\right\},
\end{gather}
and define the infimum values
$$   e^V_\eps(M) : = \inf\left\{ \EepsV(u) \, : \,  u\in\HM\right\}, \qquad
    \ezV (M) : = \inf\left\{ \EzV(u) \, : \,  u\in\XM\right\},
$$
for the constrained TFDW and liquid drop problems. In recognition of the subtlety of the existence problem for minimizers of both models (see \cite{CP}, \cite{NamVBosch}, \cite{AlamaBronsardChoksiTopalogluDroplet}, \cite{AAB}, and the excellent review article \cite{ChoksiMuratovTopaloglu}), the target space and $\Gamma$-limit must incorporate the concentration structure of minimizing sequences for the liquid drop model:  while minimizing sequences for either TFDW or liquid drop may not converge, they do concentrate at one or more mass centers, and if there is splitting of mass the separate pieces diverge away via translation.  We define the energy ``at infinity'',  $\Ezz(u)$, taking potential $V\equiv 0$, with infimum value $\ezz(M)$.  From this we then define the appropriate $\Gamma$-limit as 
\begin{align}\label{GmmLmmt}\FzV(\{u^i\}_{i=0}^\infty)&:=\begin{cases}\mathcal{E}_0^V(u^0)+\displaystyle\sum_{i=1}^\infty \mathcal{E}_0^0(u^i), &\{u^i\}_{i=0}^\infty\in \Hz, \\  \infty, & \textit{otherwise,}\end{cases}\end{align}
on the space of limiting configurations,
\begin{align}\Hz:=\left\{\{u^i\}_{i=0}^\infty\subset  \XX; \  \ \sum_{i=0}^\infty \int_{\RRR}|\nabla u^i|<\infty, \  \ \sum_{i=0}^\infty\| u^i\|_{\Ltwo}^2=M\right\}.\end{align}

%

We now state our convergence result, which is in the spirit of $\Gamma$-convergence but with respect to a notion of convergence which is suggested by Concentration-Compactness, given by \eqref{Cmpctnss}-\eqref{xfar}.

\begin{Th}\label{Th1}  
 $\Eeps^V$ $\mathit{\Gamma}-$converges to $\FzV$, in the sense that:
\begin{enumerate}[(i)]
\item (Compactness and Lower-bound) For any sequence $\ek\ktends  0^+ $, if 
 $\{u_\ek\}_{k\in\NN}\subset\HM$ and $\sup_{k}\mathcal{E}_{\ek}^V(u_\ek)<\infty$, then there exist a subsequence (still denoted $\ek$), a collection $\{u^i\}_{i=0}^\infty\in{\Hz}$, and translations 
 $\{x_k^i\}_{k\in\NN}\subset\RRR$, 
 so that
\begin{gather}\label{Cmpctnss} u_\ek(\cdot)- \left(u^0  +\sum_{i=1}^\infty u^i(\cdot-x_k^i)\right) \ktends0\textit{ in $\Ltwo$}, \\
\label{xfar}  |x_k^i|\ktends 0, \quad |  x_k^i-x_k^j | \ktends \infty,\quad i\neq j, \\
\label{LwrBnd}\FzV(\{u^i\}_{i=0}^\infty)\leq \liminf_{k\to\infty}\mathcal{E}_\ek^V(u_\ek).
\end{gather}
\item (Upper-bound) Given $\{u^i\}_{i=0}^\infty\in\Hz$ and any sequence $\ek\ktends  0^+ $, there exist functions $\{u_\ek\}_{k\in\NN}\subset\Hz$ and translations $\{x_{k}^i\}_{k\in\NN}\subset\RRR$, such that equations \eqref{Cmpctnss} and \eqref{xfar} hold, and
\begin{align}\label{UpprBnd}\FzV(\{u^i\}_{i=0}^\infty)\geq\limsup_{k\to\infty} \Eeps^V(u_\ek).\end{align}
\end{enumerate}\end{Th}

 We note that $u^0$ is the limit of $u_\eps$ in $L^2_{loc}(\RRR)$, and could well be zero.  However, it is natural to distinguish this component as it is the only one which ``feels'' the effect of $V$, and for minimizers when ($V\not\equiv 0$) it will be nontrivial.

The compactness and lower semicontinuity  (with respect to the notion of convergence given by \eqref{Cmpctnss}-\eqref{xfar}), combine two different approaches in the calculus of variations.  Local convergence of the singular limits uses BV bounds in the flavor of the Cahn-Hilliard problems, as studied in \cite{ModMort,Sternberg}.  On the other hand, the lack of global compactness imposes a concentration-compactness structure~\cite{LionsConcentrationPart1,Lions,FrankLieb,AlamaBronsardChoksiTopalogluDroplet}, in order to recover all of the mass escaping to infinity.  The proof of part (i) is done in section 2.

For the recovery sequence and upper bound (ii), the presence of an infinite number of $\Ui$ presents some obstacles not normally seen in Cahn-Hilliard-type problems, where the setting is usually a bounded domain or flat torus.  Indeed, for (ii) of Theorem~\ref{Th1} we must consider $\{ u^i\}_{i=0}^\infty$ with infinitely many nontrivial components, and then it is only possible at any fixed $\eps>0$ to construct a trial function approximating $u^i$ when the scale of its support is large compared to $\eps$.  This construction will be done in section 3. 

 While Theorem~\ref{Th1} expresses convergence of a family of variational problems in the spirit of de Giorgi's $\Gamma$-convergence, it does not fit the standard form defined in most texts on the subject, (see for example \cite{Braides}), since the topology of the convergence is not determined by the choice of a common underlying space which contains the domains of the functionals $\Eeps^V$ and $\FzV$.  More general notions of $\Gamma$-convergence have been introduced to allow for contexts in which there is no common ambient space; see \cite{JeSt} for instance.  This form of the $\Gamma$-limit, as a sum of disassociated variational problems splitting on different scales  was already introduced in droplet breakup for di-block copolymers; see \cite{CP,ABLW}.  

An important motivation behind de Giorgi's introduction of $\Gamma$-convergence was to understand the existence of, and relations between, minimizers of the functionals involved.
 In the following paragraph, we discuss the implications of our theorem to minimization problems in various settings, and the proofs of those results will be given in section 4.

\subsection*{On Minimizers}

Here we discuss the implications of the $\Gamma$-convergence result on minimizers of TFDW \eqref{TFDW} and of the liquid drop problem.  
For minimizers we note that $\EepsV(|u|)=\EepsV(u)$, $\EzV(|u|)=\EzV(u)$,  and so we restrict to the cone of nonnegative functions $\HM_+$, $\Xp$, $\Hzp$ as the domain for each.   Hence the triple-well nature of the potential $W(u)$ is not felt at all for energy minimizers, although it is an interesting question whether one can impose some constraint or min-max procedure which produces critical points which exploit the third well in a nontrivial way.

In some sense, one tends to think of a $\Gamma$-limit as a framework in which minimizers of the $\eps$ functionals should converge to minimizers of the limiting energy (see, e.g., \cite{KS}), but given the complexity of the question of the existence of minimizers for each model, this is a subtle point.
The notion of {\it generalized minimizers}, introduced for the case $V\equiv 0$ in \cite[Definition 4.3]{KMN}, provides a useful means of discussing the structure of minimizing sequences which may lose compactness:


\begin{Def}\label{genmin}  Let $V$ satisfy \eqref{HypV} and $M>0$.  A \underbar{generalized minimizer} of $\EzV(M)$ is a finite collection $\{u^0,u^1,\dots,u^N\}$, $u^i\in BV(\RRR,\{0,1\})$, such that:
\begin{enumerate}
\item  $\|u^i\|_{\Ltwo}^2:=m^i$, $i=0,1,\dots,N$, with $\sum_{i=0}^N m^i = M$;
\item  $u^0$ attains the minimum $\ezV(m^0)$ and $u^i$ attains $\ezz(m^i)$, $i=1,\dots,N$;
\item  $ \ezV(M)=  \ezV(m^0) + \sum_{i=1}^N \ezz(m^i)$.
\end{enumerate}
\end{Def}

In \cite{AlamaBronsardChoksiTopalogluDroplet} it is shown that to {\it any} minimizing sequence for the liquid drop model with (or without) potential $V$, one may associate a generalized minimizer as above.  In this way, up to translation ferrying the components $u^i$ to infinity, the collection of all generalized minimizers of $\EzV$ with constrained mass $M$ completely characterizes the minimizing sequences of $\EzV$.

We naturally associate to a generalized minimizer $\{u^0,u^1,\dots,u^N\}$ an element
$\{u^i\}_{i=0}^\infty$ of $\mathcal{H}_0^M$ by taking $u^i=0$ for all $i\ge N+1$, and then we have $\FzV (\{u^i\}_{i=0}^\infty)=  \ezV(M)$.  When convenient we abuse notation and denote $\FzV (\{u^i\}_{i=0}^N)$ the value of the limiting energy for a generalized minimizer.
We may thus address the convergence of minimizers of $\EepsV $ (should they exist) in terms of generalized minimizers of $\EzV $, using Theorem~\ref{Th1}:

\begin{Th}\label{minlim}
Let $M>0$ and assume that there exists $\eps_n\ntends 0^+ $ for which $e^V_{\eps_n}(M)$ is attained at $u_n\in \mathcal{H}_+^M$ for each $n\in\NN$.  Then, there exists a subsequence (not relabeled) and a generalized minimizer 
$\{u^0, \dots, u^N\}$ of $\EzV$ for which \eqref{Cmpctnss} and \eqref{xfar} hold for $i=0,\dots,N$, and 
$$ \FzV (\{u^i\}_{i=0}^{N})=  \ezV(M) = \lim_{n\to\infty} e^V_{\eps_n}(M).  $$
\end{Th}
A slightly more general version of Theorem \ref{minlim} will be proven in Lemma~\ref{minlem}.

There is a special class of potentials $V$ for which the existence problem $\inf \mathcal{E}_\eps^V$ is completely understood for each $\eps$; namely, $V$ of \underbar{long-range}, which are potentials that satisfy 
\begin{equation}\label{Long}
  \liminf_{t\to\infty}t\left(\inf_{ |  x  | =t}V( x )\right)=\infty.
\end{equation}
For example, the homogeneous potentials $V^\nu( x )=  |  x | ^{-\nu}$ are of long-range for $0<\nu<1$.  For  $V\in \Lthreehalves+\Linfty $ satisfying \eqref{Long} it is known that the global minimum is attained for any $M>0$, for both the TFDW and liquid drop functionals~\cite[Theorems 1 and 2]{AlamaBronsardChoksiTopalogluLongRange}.  For this class of problem, we then obtain the global convergence of minimizers in $L^2$ norm:

\begin{Cor}\label{cor:super}
Assume $V$ satisfies \eqref{HypV} and \eqref{Long}, and for $M>0$ and $\eps>0$,  let $u_\eps\in \mathcal{H}_+^M$ be a minimizer of $ \eepsV(M)$.  Then, for any sequence $\eps_n\ntends 0^+ $ there exists a subsequence (not relabeled) and a minimizer $u^0\in \Xp$ of $ \ezV(M)$ with $u_{\eps_n}\ntends u^0$ in $\Ltwo$.
\end{Cor}

The most important examples for TFDW are those with atomic or molecular potentials $V$, as they are related to the Ionization Conjecture~\cite{Lions,LeBris,FrankNamVBosch,LuOttoLiquidDropBackground,NamVBosch}.  We consider the atomic case,
\beq\label{atomic}
V(x)=V_Z(x)= {Z\over |x|},
\eeq
with $Z\ge 0$ representing a constant nuclear charge.  
With slight abuse of notation, we denote by $\Eeps^Z$, $\EzZ$ the energies \eqref{EpsFnctnl} and \eqref{LDen}, respectively, with the atomic choice $V=V_Z=Z/|x|$, and
\begin{gather}   e_\eps^Z(M):= \inf \left\{  \Eeps^Z(u) \, : \, u\in \HM_+\right\}, \quad
    \ezZ(M) := \inf \left\{  \Ez(u) \, : \,  u\in \Xp\right\}.
\end{gather}
For this choice of potential, and in the liquid drop setting, Lu and Otto~\cite{LuOttoLiquidDropBackground} proved that there exists $\mu_0>0$ for which the ball $\BRM=B_{r_M}(0)$, $r_M=\sqrt[3]{{3M\over 4\pi}}$, centered at the origin of volume $M$ is the unique (up to translations for $Z=0$), strict minimizer of $\ezZ(M)$ for all $0<M<Z+\mu_0$.  The corresponding existence result for TFDW is much weaker:  by a result of LeBris~\cite{LeBris}, for each $\eps>0$ fixed, there exists $\mu_\eps>0$ for which $e_\eps^Z(M)$ is attained for all $0<M<Z+\mu_\eps$.  A natural conjecture is that the intervals of existence converge, that is $\mu_\eps\etends \mu_0$.  Using Theorem~\ref{Th1} we are able to prove the following:

\begin{Th}\label{almostminlim}
Let $V(x)=Z/|x|$, $Z>0$.
\begin{enumerate}
\item[(a)]  For any $M\in (0, Z+\mu_\eps)$, $e_\eps^Z(M)$ is attained at $u_\eps\in\HM_+$ for each $\eps>0$, and $u_\eps\etends \mathbbm{1}_{\BRM}$ in $L^2$ norm.
\item[(b)]   For every $M\in (Z, Z+\mu_0)$ and sequence $\eps_n\ntends 0^+$, there exists a subsequence (not relabeled), and $M_n\le M$ with $M_n\ntends M$, such that $\eepsnZ(M_n)$ attains a minimizer $u_n\in\mathcal{H}_{+}^{M_\eps}$.  Moreover, $u_n\ntends \mathbbm{1}_{\BRM}$ in $L^2$ norm.
\end{enumerate}
\end{Th}

Theorem \ref{almostminlim} is connected to the classical Kohn-Sternberg~\cite{KS} result on the existence of local minimizers of the $\eps$-problem in an $L^2$-neighborhood of an isolated local minimizer of the $\Gamma$-limit. We find minimizers for $\EepsZ$ which converge to the ball of mass $M$ as $\eps\to0^+$ in $\Ltwo$, which would have the given mass $M$ except for the possibility of vanishingly small pieces splitting off and diverging to infinity as $\eps\to0^+$. If it were possible to give a uniform (in $\eps> 0$) lower bound on the quantity of diverging mass in the case of splitting, then we would be able to eliminate this possibility completely and assert that $M_\eps = M$ in (b), as conjectured above.

\section{Compactness and Lower Bound}

In this section we prove part (i) of Theorem~\ref{Th1}.  This involves combining lower bounds on singularly perturbed problems of Cahn-Hilliard type with concentration-compactness methods, to deal with possible loss of compactness via splitting.

In this section we fix a potential $V$ satisfying \eqref{HypV}.  
 Throughout the paper, we shall denote by $C$ a generic constant whose value may change from one line to another.
We begin with some preliminary estimates.

\begin{Lem}\label{TwrdsBV}  Let $\{v_\eps\}_{\eps>0}\subset H^1(\RRR)$, with $\|v_\eps\|_{\Ltwo}^2\le M$ and $\EepsV (v_\eps)\le K_0$, where $K_0>0$ is a constant independent of $\eps$.  Then there exists a constant $C_0=C_0(K_0,M,V)$ such that $\forall 0<\eps<\frac{1}{4}$, we have
\begin{align}\label{Fepsbdd}\int_{\RRR}\left[\frac{\eps}{2} | \nabla v_\eps | ^2+\frac{1}{2\eps}W(v_\eps)\right]dx + D(\vert v_\eps\vert^2,\vert v_\eps\vert^2)
\leq C_0.\end{align}
\end{Lem}
\begin{proof} First by \eqref{HypV}, we write $V=V_{5/2}+V_{\infty}$, where $V_{5/2}\in L^{\frac{5}{2}}(\RRR)$ and $V_{\infty}\in\Linfty$, and fix $K>0$ large enough so that
\begin{align} |  t | ^{\frac{10}{3}}\leq \frac{5}{3}W(t),\quad  |  t | > K .\end{align}
Then, by Young's inequality, for any $u\in H^1(\RRR)$, 
\begin{align}\int_{\RRR}V\vert u\vert^2dx&\leq\int_{\RRR}V_{5/2}\vert u\vert^2dx+\| V_\infty\|_{\Linfty}\int_{\RRR}\vert u\vert^2dx\\
&\leq \frac{2}{5}\int_{\RRR} |  V_{5/2} | ^{\frac{5}{2}}dx+\frac{3}{5}\int_{\RRR} |  u | ^{\frac{10}{3}}dx+\| V_\infty\|_{\Linfty}\int_{\RRR}\vert u\vert^2dx\\
&\leq C\left(1+\int_{\{|  u | <K\}} | u|^{2}dx\right)+\int_{\{|  u | >K\}}W(u)dx+\| V_\infty\|_{\Linfty}\int_{\RRR}\vert u\vert^2dx\\
&\leq C_2+C_1\int_{\RRR}\vert u\vert^2dx+\frac{1}{2\eps}\int_{\RRR}W(u)dx.\end{align}
Hence, there exist constants $C_1,C_2>0$ for which
\begin{equation}\label{eqnCrcvty}2\Eeps^V(u)+C_1\int_{\RRR}\vert u\vert^2dx+C_2  \geq\int_{\RRR}\left[\frac{\eps}{2} | \nabla u | ^2+ \frac{1}{2\eps}W(u)\right]dx+D(\vert u\vert^2,\vert u\vert^2),
\end{equation}
and the desired estimate follows.
\end{proof}


\begin{Rmk}\label{RmkTwrdsBV} Under the hypotheses of Lemma~\ref{TwrdsBV},  
$\{v_\eps\}_{\eps>0}$ is bounded in $\Ltenthirds$ and 
\begin{align}\label{limWis0} \int_{\RRR}W(v_\eps)dx\etends 0.\end{align}
\end{Rmk}

\begin{Lem}\label{Vlimit}
	Assume $V$ satisfies \eqref{HypV}, and $\{u_n\}_{n\in\NN}, \{v_n\}_{n\in\NN}$ are sequences which are bounded in $\Ltwo \cap L^{\frac{10}{3}}(\RRR)$ and such that  $(u_n - v_n)\ntends 0$ in $\Ltwoloc $.  Then,
$$  \int_{\RRR}V\left( |u_n|^2 - |v_n|^2\right) dx  \ntends 0.  $$  
\end{Lem}

\begin{proof}  Let $\delta>0$ be given.
By \eqref{HypV} we may decompose $V=V_1+V_2+V_3$, where:
$$  V_1(x)=V(x)[1-\mathbbm{1}_{B_R}(x)], \quad
   V_2(x)=[V(x)-t]_+\mathbbm{1}_{B_R}(x), \quad
     V_3(x)=\min\{V(x) ,t\}\mathbbm{1}_{B_R}(x),
$$
with  $R$ large enough that $\|V_1\|_{\Linfty} < \delta$; $t$ large enough that $\|V_2\|_{L^{\frac52}(\RRR)}<\delta$.  Note that $V_3$ is compactly supported and uniformly bounded.  We then consider each part separately:
\begin{gather}
\int_{\RRR} V_1  \bigl| |u_n|^2-|v_n|^2\bigr| dx  \le  \delta (\|u_n\|_{\Ltwo}^2 + \|v_n\|_{\Ltwo}^2) \le c\delta; \\
\int_{\RRR} V_2 \bigl| |u_n|^2-|v_n|^2\bigr|  dx  
   \le  \|V_2\|_{L^{\frac52}(\RRR)} ( \|u_n\|_{\Ltenthirds} + \|v_n\|_{\Ltenthirds}) \le c\delta;  \\
\int_{\RRR}  V_3  \bigl| |u_n|^2-|v_n|^2\bigr| dx  \le \|V_3\|_{\Linfty} \int_{B_R} 
    \bigl| |u_n|^2-|v_n|^2\bigr| dx   \ntends 0.
\end{gather}
As $\delta>0$ is arbitrary, the result follows.
\end{proof}

\begin{Rmk}
The hypothesis \eqref{HypV} is slightly stronger than is typical for problems of TFDW type, in which a weaker local integrability is assumed, $V\in{\Lthreehalves}+\Linfty $.  (See e.g., \cite{BenguriaBrezisLieb,NamVBosch}.)  Having $V\in \Lthreehalvesloc$ is a natural condition for using the squared gradient to control $V|u|^2$ via the Sobolev embedding.  However, given the singularly perturbed nature of $\EepsV$, control on the Dirichlet energy is lost as $\eps\to 0^{+}$, and we must rely on the $L^{\frac{10}{3}}$ norm instead; hence the need for the more stringent ${\Lfivehalves}+\Linfty$ demanded in \eqref{HypV}. 
\end{Rmk}

\medskip

Next, we prepare the way for the proof of the compactness part of Theorem~\ref{Th1} by establishing that sequences $\{u_\eps\}_{\eps>0}$ with bounded energy must have centers of concentration, even if they are divergent.  The following Lemma will be used to rule out dissipation of $\{u_\eps\}_{\eps>0}$ as long as the BV norm is bounded and the $ L^{\frac43 }$ norm of $u_\eps$ is not vanishing:

\begin{Lem}\label{inqltyBV}There exists a universal constant $C>0$ such that for all $\psi\in \BV,$
\begin{equation}\label{NoVnshngCnd}
\| \psi\|_{\BV}\left[\sup_{a\in\RRR}\int_{B_1(a)} |  \psi |  dx\right]^\frac{1}{3}\geq C\int_{\RRR} |  \psi | ^{\frac43 } dx.
\end{equation}
\end{Lem}
\begin{proof} It suffices to prove \eqref{NoVnshngCnd} holds for $\psi\in W^{1,1}(\RRR)$, as we can extend it to $\psi\in \BV$ by using a density argument~\cite[Theorem 3.9.]{AmbrosioFuscoPallaraBV}.

Let $\psi\in W^{1,1}(\RRR)$, and define $\chi_a:=\chi(x-a)$, where  $\chi\in C_0^\infty(\RRR)\setminus\{0\}$ is any nonnegative function that is compactly supported in $B_1(0)$.

Then, by H\"{o}lder's inequality and Sobolev's inequality,
\begin{align}\int_{B_1(a)} |  \chi_a\psi | ^{\frac43 } dx&=\int_{B_1(a)}|  \chi_a\psi |^{\frac13} |  \chi_a\psi |  dx\\
& \leq   \left[{\int_{B_1(a)} | \chi_a\psi |  dx}\right]^\frac{1}{3}\left(\int_{\RRR} |  \chi_a\psi | ^{\frac{3}{2}}dx\right)^{\frac23}\\
&\leq C\left[{\sup_{a\in\RRR}\int_{B_1(a)} |  \psi |  dx}\right]^\frac{1}{3}\int_{\RRR} | \nabla(\chi_a\psi) |  dx\\
&\le C\left[{\sup_{a\in\RRR}\int_{B_1(a)} |  \psi |  dx}\right]^\frac{1}{3}\int_{\RRR}\left(\chi_a | \nabla \psi | + |  \nabla\chi_a |   |  \psi |  \right)dx.\end{align}
We conclude the proof of this Lemma by integrating with respect to $a\in\RRR$.\end{proof}

From this Lemma we may then conclude that noncompactness of sequences with bounded $\BV$ norm is due to splitting and translation.  The following is an adaptation of \cite[Proposition 2.1]{FrankLieb}, which is proven for characteristic functions of finite perimeter sets.

\begin{Prop}\label{FL2.1}
Assume $\{\psi_n\}_{n\in\NN}$ is a bounded sequence in $\BV$, for which $\liminf_{n\to\infty}\|\psi_n\|_{\Lfourthirds}>0$.  Then, there exists translations $\{a_n\}_{n\in\NN}\subset\RRR$, and $\psi^0\in \BV$,  $\psi^0\not\equiv 0$, such that for some subsequence (not relabeled) we have:
\begin{enumerate}
\item[(a)] $\psi_n(\cdot - a_n) \ntends \psi^0$ in $\Loneloc $,
\item[(b)]  $\| \psi^0 \|_{\BV} \le \liminf_{n\to\infty} \| \psi_n \|_{\BV}$.
\end{enumerate}
\end{Prop}

\begin{proof}
By Lemma~\ref{inqltyBV}, we have
$$  \sup_{a\in\RRR}\int_{B_1(a)}\vert \psi_n\vert dx 
 \ge \left[C{\int_{\RRR}\vert \psi\vert^{\frac{4}{3}} dx \over \|\psi_n\|_\BV}\right]^3 \ge 2c,  $$
for some $c>0$ independent of $n$.  Hence, for each $n\in\NN$ we may choose $a_n\in\RRR$ for which
\beq\label{not0}
\int_{B_1(a_n)}\vert \psi_n\vert dx\ge c>0. 
\eeq
 As $\{\psi_n(\cdot-a_n)\}_{n\in\NN}$ is bounded in $\BV$, there exists a subsequence and $\psi^0\in\BV$ for which (a) and (b) hold.  By \eqref{not0} and $L^1_{loc}$ convergence, the limit $\psi^0\not\equiv 0$.
\end{proof}
Once we have localized a piece of our $\BV$-bounded sequence $\{\psi_n\}_{n\in\NN}$ as an $L^1_{loc}$-converging part, we will need to separate the compact piece from the rest, which converges locally to zero but may carry nontrivial $L^1$-mass to infinity.  To do this, we first define a smooth cut-off function $\omega:\mathbb{R}\to[0,1]$, with
\begin{align}\label{DfOmg}\omega\equiv1\text{ for } \ x<0,\quad  \omega\equiv0\text{ for } \ x>1,\text{ and }  \|\omega'\|_{\Linfty}\leq2,
\end{align}
and for any $\rho>0$, 
\beq\label{OmgEps}
\omega_\rho(x)= \omega( |x|-\rho ).
\eeq

The next Proposition is based on \cite[Lemma 2.2.]{FrankLieb}:

\begin{Prop}\label{CmpctnssBck}Let $\{\psi_n\}_{n\in\NN}$ be bounded in $\BV$ with $\psi_n\ntends\psi^0$ in $\Loneloc$ and pointwise almost everywhere in $\RRR$, for some function $\psi^0\in\BV$. If $0<\|\psi^0\|_{\Lone}<\liminf_{n\to\infty}\| \psi_n\|_{\Lone}$, then there exist radii $\{\rho_n\}_{n\in\NN}\subset(0,\infty)$ such that, up to a subsequence,
\begin{align}\label{SpltTtlVr} \int_{\RRR}[ | \nabla\psi_n | - | \nabla(\psi_n\omega_{\rho_n}) | - | \nabla(\psi_n-\psi_n\omega_{\rho_n}) | ]\ntends 0.\end{align}
Moreover,
\begin{align}\label{Prts}\psi_n\omega_{\rho_n}\ntends\psi^0 \ \textit{ in $\Lone$ and }
  \psi_n(1-\omega_{\rho_n})\ntends 0 \ \textit{ in $\Loneloc$,}\end{align}
with each converging pointwise almost everywhere in $\RRR$.
\end{Prop}

\begin{proof}
Note that
\begin{align}\int_{\RRR} | \nabla\psi_n |  
&\leq \int_{\RRR} |\nabla (\psi_n\omega_{\rho_n})| 
+\int_{\RRR} | \nabla\left[ \psi_n(1-\omega_{\rho_n})\right] |   \\
&\leq \int_{\RRR} | \nabla\psi_n |   +2\int_{\RRR}\left |  \psi_n\nabla\omega\left(|x|-\rho_n\right)\right |  dx\\
&\leq \int_{\RRR} | \nabla\psi_n |  +4\int_{B_{\rho_n+1}(0)\setminus B_{\rho_n}(0)} |  \psi_n |dx.\end{align}
Therefore, \eqref{SpltTtlVr} will hold if we find $\{\rho_n\}_{n\in\NN}\subset(0,\infty)$ such that 
\begin{align}\label{smallannulus}
 \int_{B_{\rho_n+1}(0)\setminus B_{\rho_n}(0)} |  \psi_n |  dx\ntends 0.\end{align}

We distinguish between two cases. First, suppose that $\text{supp}\, \psi^0\subset B_R(0)$, for some $R>0$. In this case, we claim that it suffices to choose $\rho_n=R$ for all $n\in\NN$.
Indeed, by $\Loneloc$ convergence and the compact support of $\psi^0$,  
\begin{align} \left | \left |  \psi_n\omega_{\rho_n}\right | \right | _{\Lone}= \left | \left |  \psi_n\omega_{\rho_n}\right | \right | _{ L^1(B_{R+1}(0))}\ntends \|\psi^0\|_{\Lone}.\end{align}
Therefore  each sequence converges pointwise a.e., and \eqref{Prts} holds by the Brezis-Lieb Lemma~\cite{BrezisLieb}. Also, since $\psi^0\mathbbm{1}_{B_{R+1}(0)\setminus B_{R}(0)}\equiv0$ and $\psi_n\ntends \psi^0$ in $ L^1(B_{R+1}(0))$, we conclude that \eqref{smallannulus} is also verified in case $\text{supp}\, (\psi^0)$ is compact.

In the second case, if $\text{supp}\, \psi^0$ is essentially unbounded, note that 
$\|\psi^0\|_{\Lone}<\liminf_{n\to\infty}\| \psi_n\|_{\Lone}$ 
implies that along some subsequence (not relabeled) we may choose $R_n$ such that  
\beq\label{FL1}
\int_{B_{R_n}(0)} |\psi_n| dx  =\|\psi^0\|_{\Lone}.
\eeq  

We claim that, chosen this way,
$R_n\ntends \infty$.  Indeed,  if (along a further subsequence if necessary) we had $R_n\le R_0$ for some $R_0>0$, it would follow from the $L^1_{loc}$ convergence that,
\begin{align}
\|\psi^0\|_{\Lone}=
\liminf_{n\to\infty}\| \psi_n\One_{{B}_{R_n}}\|_{\Lone}
\le \liminf_{n\to\infty} \|\psi_n\One_{B_{R_0}}\|_{\Lone}
=\|\psi^0\One_{B_{R_0}}\|_{ \Lone}<\|\psi^0\|_{\Lone},
\end{align}
since we are assuming that $\text{supp}\, \psi^0$ is essentially unbounded.  Thus, $R_n\ntends \infty$.

Next, fix $R>1$ such that
$$   \int_{B_R(0)} |\psi^0| dx   \ge \frac12 \|\psi^0\|_{\Lone}.  $$
By $\Loneloc$ convergence, for all sufficiently large $n$ we have
\beq\label{FL2}
 \int_{B_R(0)} |\psi_n| dx   \ge \frac14 \|\psi^0\|_{\Lone}.
\eeq

We now claim that for $n$ large enough such that $R_n>R$, there exists $\rho_n\in \left[{R+R_n\over 2}, R_n\right]$ for which
\beq\label{sa2}   \int_{B_{\rho_n+1}(0)\setminus B_{\rho_n}(0)} |\psi_n| dx   \le {3\over R_n-R} \|\psi^0\|_{\Lone}.  
\eeq
If so, then \eqref{smallannulus} is satisfied with this choice of $\rho_n\ge r_n:={R+R_n\over 2}\ntends  \infty$.
To verify the claim, suppose the contrary, and so for every $\rho\in \left[r_n, R_n\right]$ we have the opposite inequality to \eqref{sa2}.  For fixed $n$, choose $K\in\NN$ with 
$R_n-1\le r_n +K < R_n$, so there are $K$ intervals of unit length lying in 
$\left[r_n, R_n\right]$.  Then, by \eqref{FL1}, \eqref{FL2},
\begin{align}
\frac34 \|\psi^0\|_{\Lone} &\ge  \int_{B_{R_n}(0)} |\psi_n|  dx  - \int_{B_R(0)} |\psi_n| dx   \ge
      \int_{B_{r_n+K}(0)\setminus B_{r_n}(0)} |\psi_n| dx    \\
      &> K{3\over R_n-R} \|\psi^0\|_{\Lone} \ge 3{R_n -r_n-1\over R_n-R}\|\psi^0\|_{\Lone}=\frac32{R_n-R-2\over R_n-R}\|\psi^0\|_{\Lone},
\end{align}
for all sufficiently large $n$, a contradiction.  This completes the proof of \eqref{SpltTtlVr}.

To complete the proof of Proposition~\ref{CmpctnssBck}, first  note that  (up to a subsequence), $\psi_n\omega_{\rho_n}\ntends \psi^0$ almost everywhere in $\RRR$, and recall that $\rho_n\leq R_n$.  Hence, from \eqref{smallannulus}, \eqref{FL1} we obtain:
\begin{align} \|\psi^0\|_{\Lone}&\leq\liminf_{n\to\infty}\int_{\RRR}|\psi_n \omega_{\rho_n}|dx\\
&=\liminf_{n\to\infty}
   \left[\int_{B_{\rho_n}(0)} |\psi_n|dx+ \int_{B_{\rho_n+1}(0)\setminus B_{\rho_n}(0)} |\psi_n|dx\right]\\
&=\liminf_{n\to\infty}\int_{B_{\rho_n}(0)} |\psi_n|dx
\leq \liminf_{n\to\infty}\int_{B_{R_n}(0)} |\psi_n|dx =
\|\psi^0\|_{\Lone}.\end{align}
Thus each inequality above is an equality, $\|\psi_n\omega_{\rho_n}\|_{\Lone}\ntends \|\psi^0\|_{\Lone}$, and hence \eqref{Prts} follows from the Brezis-Lieb Lemma~\cite{BrezisLieb}.
\end{proof}

\begin{Rmk}\label{RmkCmpctnssBck}  By lower semicontinuity of the total variation with respect to $L^1$ convergence, up to a subsequence,
\begin{align}\int_{\RRR} | \nabla \psi^0 |  
  \leq \lim_{n\to\infty}\int_{\RRR}\left( | \nabla\psi_n |  
    -  | \nabla(\psi_n-\psi_n\omega_{\rho_n}) | \right)
  \end{align}\end{Rmk}

We are now ready to prove the compactness and $\Gamma$-liminf part of the theorem:
\begin{proof}[Proof of Theorem~\ref{Th1} (i)]   Let $\{u_\eps\}_{\eps>0}$ be a family in $\HM$ with $\EepsV (u_\eps)\le K_0$, $\eps>0$.  

\noindent {\bf Step 1:}  {\it Truncation.}  

First, we show that when proving (i) it suffices to restrict to $u_\eps$ satisfying the pointwise bounds $-1\le u_\eps \le 1$ almost everywhere in $\RRR$.  Indeed, we define the truncations
\begin{align}\label{newustar}u_\eps^*:=\begin{cases}-1,&u_\eps<-1,\\
u_\eps,&  |  u_\eps |  \leq 1,\\
1,&u_\eps>1.\end{cases}\end{align}
We will show that  $\| u_\eps -u^*_\eps\|_{\Ltwo}^2\etends 0$, and
\begin{align}\label{truncated}
\liminf_{\eps\to 0^+}\EepsV (u_\eps^*)\leq\liminf_{\eps\to 0^+}\EepsV (u_\eps).
\end{align}

To accomplish this, we first note that by Remark \ref{RmkTwrdsBV}, we have that
\begin{align}0\leq \int_{\RRR} |  u_\eps-u_\eps^* | ^{2} dx
= \int_{\{ |  u_\eps | >1\}}\left( |u_\eps|- 1\right)^{2} dx
\leq C \int_{\RRR}W(u_\eps)dx \etends0,
\end{align}
where $C$ is a constant independent of $\eps$.    Also by Remark~\ref{RmkTwrdsBV},  $\{u_\eps\}_{\eps>0}$ is bounded in ${\Ltwo}\cap\Ltenthirds$, and hence the sequence of truncations $\{u^*_\eps\}_{\eps>0}$ is as well. By Lemma~\ref{Vlimit}, we conclude that the local potential terms are close,
$$  \int_{\RRR} V  \left( |u_\eps|^2 - |u^*_\eps|^2\right) dx   \etends 0.  $$ 
Finally, each of the other terms decreases under truncation,
$$   |\nabla u^*_\eps |\le |\nabla u_\eps |, \quad
         W(u^*_\eps )\le W(u_\eps ), \quad  D(\vert u^*_\eps\vert^2,\vert u^*_\eps\vert^2)\le D(\vert u_\eps\vert^2, \vert u_\eps\vert^2), $$
and so \eqref{truncated} is verified.

In the following we will therefore assume without loss of generality that $-1\le u_\eps \le 1$, $\eps>0$, almost everywhere in $\RRR$.

\textbf{Step 2:} {\it Passing to the first limit.} 

Let $\phi_\eps:=\Phi(u_\eps)$, where $\Phi:\R\to\R$ is defined by
\begin{align}\label{defphi}\Phi(t):=\int_0^t\sqrt{W(\tau)}d\tau.\end{align} 

Then,
\begin{align}\phi_\eps=\int_{0}^{u_\eps } |  t | (1- |  t | ^{\frac23}) dt=sign(u_\eps )\left(\frac12  |  u_\eps  | ^2-\frac{3}{8} |  u_\eps  | ^{\frac{8}{3}}\right),\end{align}
and  since $\| u_\eps \|_\Linfty\leq 1$,
\begin{align}\label{phiAndu}\frac{1}{8} |  u_\eps  | ^2\leq  | \phi_\eps | \leq\frac12  |  u_\eps  | ^2\quad \text{ and }\quad  | \phi_\eps | \leq\phi_\eps(1)=\frac{1}{8}.\end{align}
In particular, $\|\phi_\eps\|_{\Lone} \le \frac12\|u_\eps \|_{\Ltwo}^2\le {M\over 2}$.  Furthermore, $\{\phi_\eps\}_{0<\eps<\frac14}$ is  bounded in $BV(\RRR)$.  Indeed,  following \cite{ModMort}, by Young's inequality and Lemma \ref{TwrdsBV} with $v_\eps=u_\eps$,
\begin{align}\begin{split}\label{enrgphiu}\int_{\RRR} | \nabla\phi_\eps |  dx&=\int_{\RRR}\sqrt{W(u_\eps )} | \nabla u_\eps  |  dx\leq\int_{\RRR}\left[\frac{\eps}{2} | \nabla u_\eps | ^2+\frac{1}{2\eps}W(u_\eps)\right]dx
\le  K_1,
\end{split} \end{align}
with constant $K_1=K_1(K_0, M,V)$.
Consequently, $\{\|\phi_\eps\|_{\BV}\}_{0<\eps<\frac14}$ is bounded.

Now let $\ek\ktends  0^+ $ be any sequence.  
By the compact embedding of $\BV$ in $\Loneloc$ there exist a subsequence, which we continue to denote by $\ek\ktends  0^+ $, and a function $\phi^0\in\BV$ so that $\phi_{\ek}\ktends \phi^0$ in $\Loneloc$ and almost everywhere in $\RRR$. Moreover, by lower semicontinuity of the total variation,
\begin{align}\label{lscGrad}\int_{\RRR} | \nabla\phi^0 |  \leq\liminf_{k \to\infty}\int_{\RRR} | \nabla\phi_{\ek} |  dx.\end{align} 
Now we can use the invertibility of $\Phi$ and the local uniform continuity of $\Phi^{-1}$ to obtain that $u_{\ek} \ktends  u^0:=\Phi^{-1}(\phi^0)$ almost everywhere in $\RRR$. Then, by Fatou's Lemma and Remark~\ref{RmkTwrdsBV}, we have
\begin{align}0\leq\int_{\RRR}W(u^0)dx\leq\liminf_{k \to\infty}\int_{\RRR}W(u_{\ek})dx=0\end{align}
hence $W(u^0)\equiv0$, $u^0(x)\in \{0,\pm 1\}$ almost everywhere, and 
\begin{align}\label{phi0u0same}\phi^0=\frac{1}{8}u^0\textit{ almost everywhere in $\RRR$}.\end{align} 
As a result, by Fatou's Lemma and \eqref{phiAndu}, for any compact $K\subset\RRR$,
\begin{align}\int_{K} | \phi^0 |  dx=\frac{1}{8}\int_K |  u^0 | ^2 dx\leq \frac{1}{8}\liminf_{k \to\infty}\int_K |  u_{\ek}  | ^2dx\leq\lim_{k \to\infty}\int_{K} | \phi_{\ek} |  dx=\int_{K} | \phi^0 |  dx.
\label{strong}
\end{align}
Thus,  (taking a further subsequence if necessary), $u_{\ek} \ktends  u^0$ pointwise almost everywhere in $\RRR$.   By the Brezis-Lieb Lemma~\cite{BrezisLieb},  we obtain convergence in $\Ltwoloc$, with $\|u^0\|_{\Ltwo}^2\le M$.  Furthermore,  by Fatou's Lemma,
\beq\label{nonloc0}
D(\vert u^0\vert^2 ,\vert u^0\vert^2 ) \le \liminf_{k\to\infty}  D(\vert u_\ek\vert^2 ,\vert u_\ek\vert^2 ),
\eeq 
and  by  Lemma \ref{Vlimit}  (with $u_n=u_{\eps_k}$ and $v_n=u^0$),  \eqref{lscGrad}, \eqref{enrgphiu}, and \eqref{phi0u0same} we have
$$   \EzV(u^0)\le \liminf_{k\to\infty} \EkV(u_\ek).  $$

If  $\phi_{\ek} \ktends  \phi^0$ in the $L^1$ norm, then by the same argument as \eqref{strong} we may conclude that $u_{\ek}\ktends  u^0$ converges in $L^2$ norm, and so 
$m^0:=\|u^0\|_{\Ltwo}^2 =M$, and setting $u^i\equiv 0$ for all $i\ge 1$, the proof is complete.

\medskip

\textbf{Step 3:} {\it Splitting off the remainder sequence.}  \ 
 If $m^0=M$, then $u_{\eps_k}\ktends u^0$ in $\Ltwo$ by the Brezis-Lieb Lemma~\cite{BrezisLieb}, and setting $u^i\equiv 0$ for all $i\ge 1$, the proof is complete. To continue we assume that  
$m^0:=\|u^0\|_{\Ltwo}^2<M$, so the first limit does not capture all of the mass in the sequence $u_{\ek}$.   In this case, both $u_\ek$ and $\phi_{\ek}$  converge only locally (and not in norm), that is  
\begin{align}\label{philoc}\|\phi^0\|_{\Lone}<\liminf_{k \to\infty}\|\phi_{\ek}\|_{\Lone},
\end{align} 
and similarly for $u_\ek$, by the Brezis-Lieb Lemma~\cite{BrezisLieb}.

Applying Proposition \ref{CmpctnssBck} and Remark~\ref{RmkCmpctnssBck} to 
$\phi_{\ek}$, and the fact that we do not have global convergence, there exists a sequence of radii 
$\{\rho_k\}_{k\in\NN}\subset(0,\infty)$ with $\rho_k\ktends\infty$ so that, for  
$$   \phi^0_{\ek}:= \omega_{\rho_k}\phi_{\ek}, \quad \phi^1_{\ek}:= (1-\omega_{\rho_k})\phi_{\ek},
$$
where $\omega_\rho$ is defined in \eqref{OmgEps},
and for a subsequence (which we continue to write as $\ek\ktends  0^+ $),
\begin{gather}\label{SpltPhi}\phi^0_{\ek}\ktends \phi^0 
\quad\textit{ in $\Lone$},\quad 
\phi^1_{\ek}\ktends  0\textit{ in $\Loneloc$,}\\
\label{SpltPhi2}\phi^0_{\ek}\ktends\phi^0\text{ and }\phi^1_{\ek}\ktends0\text{ pointwise almost everywhere in $\RRR$, and}\\
\label{SpltGrdPhi}
\int_{\RRR} | \nabla\phi^0 | 
   +\int_{\RRR}\left | \nabla\phi_{\ek}^1\right |  dx\leq
   \int_{\RRR}\left | \nabla\phi_{\ek}\right |  dx  + o(1).
\end{gather}
Moreover, from \eqref{smallannulus} and \eqref{phiAndu} the mass contained in the cut-off region is negligible: 
\beq\label{smallann}  \lim_{k\to\infty}\int_{B_{\rhk+1}(0)\setminus B_{\rhk}(0)} |\phi_\ek| dx   =0
=  \lim_{k\to\infty}\int_{B_{\rhk+1}(0)\setminus B_{\rhk}(0)} |u_\ek|^2 dx  .
\eeq
We also decompose $u_\ek$ into two pieces, 
\beq\label{usplit}  u^0_\ek=u_\ek\sqrt{\omega_\rhk}, \ \text{and} \ \ u^1_\ek=u_\ek\sqrt{1-\omega_\rhk},
\eeq
so that $(u_\ek)^2 = (u^0_\ek)^2 + (u^1_\ek)^2$ and by the proof of Proposition~\ref{CmpctnssBck}

 $u^1_\ek\ktends0$ almost everywhere in $\RRR$.  Note that $\phi^i_\ek=\Phi(u^i_\ek)$ holds in $\RRR\setminus \{\rhk<|x|<\rhk+1\}$, and by \eqref{smallann} the region where they are no longer explicitly related carries a negligible amount of the mass of $u_\ek$.

Equations \eqref{SpltGrdPhi}, \eqref{phi0u0same} and \eqref{enrgphiu} give
\begin{align}\label{SpltEps}
\frac{1}{8}\int_{\RRR} | \nabla u^0 | +\lim_{k \to\infty}\int_{\RRR} | \nabla\phi_{\ek}^1 |  dx\leq \liminf_{k \to\infty}\int_{\RRR}\left[\frac{\eps}{2} | \nabla u_\eps | ^2+\frac{1}{2\eps}W(u_\eps)\right]dx\le K_0,
\end{align}
and in particular, $\{\phi_\ek^1\}_{k\in\NN}$ is bounded in $\BV$.  
The nonlocal term also splits in the same way. Indeed, by \eqref{usplit}, $u^0_\ek\ktends u^0$ pointwise almost everywhere in $\RRR$, the positivity of $D(f,g)$ for $f,g\ge 0$, and \eqref{nonloc0}
\begin{align}\begin{split}\label{SpltNnlc}
\liminf_{k\to\infty}D(\vert u_{\ek}\vert^2,\vert u_{\ek}\vert^2)&=
  \liminf_{k\to\infty}  D( |u^0_\ek|^2 + |u^1_\ek|^2, |u^0_\ek|^2 + |u^1_\ek|^2)   \\
 &\ge \liminf_{k\to\infty}  D( |u^0_\ek|^2,  |u^0_\ek|^2) +  D( |u^1_\ek|^2, |u^1_\ek|^2)  \\
 &\ge  D( |u^0|^2,  |u^0|^2) + \liminf_{k\to\infty}D( |u^1_\ek|^2, |u^1_\ek|^2).
\end{split}
\end{align}
 Moreover, \eqref{phi0u0same}, Fatou's Lemma, \eqref{phiAndu}, and \eqref{SpltPhi} give
\begin{align}\int_{\RRR} | \phi^0 |  dx&=\frac{1}{8}\int_{\RRR} |  u^0 | ^2dx
\leq\frac{1}{8}\liminf_{k \to\infty}\int_{\RRR} |  u_{\ek}^0 | ^2 dx\\
&\leq\lim_{k \to\infty}\int_{\RRR} |  \phi_{\ek}^0|  dx=\int_{\RRR} | \phi^0 |  dx.\end{align}
thus (taking a further subsequence if necessary),
$u^0_{\ek}\ktends u^0$ in $\Ltwo$.
As a result,
\begin{align}\label{SpltMss}M=m^0+\lim_{k \to\infty} M^1_\ek, \quad\text{where}\quad
M^1_\ek:=\| u_{\ek}^1\|_{\Ltwo}^2=\| u_{\ek}-u^0\|_{\Ltwo}^2 + o(1).
\end{align}
Lastly, as $u_\ek \ktends  u^0$ in $\Ltwoloc$, by Lemma~\ref{Vlimit} we have
$$  \int_{\RRR} V  |u_\ek|^2 dx  =  \int_{\RRR} V  |u^0|^2 dx  + o(1),  $$
and hence  by \eqref{phi0u0same} and \eqref{SpltEps}, we conclude 
\begin{align}\label{SpltEnrg}\Ez^V(u^0)+\liminf_{k \to\infty}\left[\int_{\RRR} | \nabla \phi_{\ek}^1 |  dx+D\left(\vert u_{\ek}^1\vert^2,\vert u_{\ek}^1\vert^2\right)\right] \leq\liminf_{\eps\to0^+}\Eeps^V(u_\eps).\end{align}


\noindent {\bf Step 4:}  \textit{Concentration in the remainder sequence.  }

For any bounded sequence $\{\psi_k\}_{k\in\NN}$ in $\Lone $ we define 
\begin{align}
\MM(\{\psi_k\}):=
\sup\{\| \psi\|_{\Lone}:\exists x_k\in\RRR,\psi_k(\cdot+x_k)\ktends  \psi \textit{ in }\Loneloc\},\end{align}
So $\MM(\{\psi_k\})$ identifies the largest possible $L^1_{loc}$ limiting mass of the sequence, up to translation.  

We claim that for our remainder sequence, $\MM(\{\phi^1_{\ek}\})>0$.  Indeed, this will follow from Proposition~\ref{FL2.1} once we have established the hypotheses.  We first note that by \eqref{SpltEps}, $\{\phi_{\ek}^1\}_{k\in\NN}$ is bounded in $\BV$. 
Next, we must show that the $L^{\frac43 }$ norm of $\phi_{\ek}^1$ is bounded below.  As $u_\ek^1 =u_\ek $ almost everywhere in $\RRR\setminus B_{\rhk+1}(0)$, from Lemma \ref{TwrdsBV} we have
$$  {4}C_0\epk \ge \int_{\RRR\setminus B_{\rhk+1}(0)}  W(u_\ek) dx   = \int_{\RRR\setminus B_{\rhk+1}(0)} W(u^1_\ek) dx  , $$
and thus, from \eqref{phiAndu}, \eqref{smallannulus}, \eqref{SpltMss}, and $t^{\frac83} = (t^{\frac{10}{3}}+t^2)/2 - W(t)/2$,  we have:
\begin{align}\begin{split}\label{L43}
  \int_{\RRR\setminus B_{\rhk+1}(0)}|\phi_\ek^1|^{\frac43 } dx   &\ge {1\over 16} \int_{\RRR\setminus B_{\rhk+1}(0)} |u^1_\ek|^{\frac83 } dx   \\
&\ge  \frac1{32}  \int_{\RRR\setminus B_{\rhk+1}(0)} \left( |u_\ek|^{\frac{10}{3}} + |u_\ek|^2 \right) dx   - 2C_0\ek  \\  &> {1\over 32} \int_{\RRR\setminus B_{\rhk}(0)} |u_\ek|^2 dx   - o(1)\\
& \ge {1\over 32}\int_{\RRR} |u_\ek^1|^2 dx   +o(1) = {M^1_{\ek }\over 32}+o(1)= \frac{1}{32}(M-m^0) +o(1) > 0.  \end{split} \end{align}
Applying Proposition~\ref{FL2.1} the claim follows.

By the claim and Proposition~\ref{FL2.1}, we may choose a subsequence, translations $\{x_k^1\}_{k\in\NN}$, and $\phi^1\in \BV$ with
 $$  \phi_\ek^1(\cdot-x_k^1) \ktends  \phi^1 \ \text{ in $\Loneloc$, and } \ \ \|\phi^1\|_{\Lone} \ge \frac12\MM(\{\phi^1_\ek\}).
 $$ 
Note that since $\phi_\ek^1\ktends  0$ in $\Loneloc$, the sequence $|x_k^1|\ktends \infty.$
By the same arguments as in Step 1 we may conclude that $u^1_\ek(\cdot-x_k^1) \ktends  u^1=8\phi^1$ in $\Ltwoloc$ and almost everywhere in $\RRR$, with $W(u^1 ) \equiv0$ almost everywhere in $\RRR$, and hence $u^1\in\XX$ with $\|u^1\|_{\Ltwo}^2= :  m^1 \le (M-m^0)$.  

Finally, the nonlocal term, which splits as in \eqref{SpltNnlc}, passes to the limit using Fatou's Lemma,
\begin{align}\label{SplitNL2}
D(|u^0|^2,(|u^0|^2) + D(|u^1|^2,|u^1|^2)
&\le  
D(|u^0|^2,|u^0|^2)+\liminf_{k \to\infty}D\left(|u^1_\ek|^2,|u^1_\ek|^2\right) \\
&\leq\liminf_{k \to\infty}D(\vert u_{\ek}\vert^2,\vert u_{\ek}\vert^2).
\end{align}

In conclusion,  using the previous inequality and \eqref{SpltEps}, we have
$$  \EzV(u^0) + \Ezz(u^1) \le \EzV(u^0) + \liminf_{k \to\infty}\left[\int_{\RRR} | \nabla \phi_{\ek}^1 |  dx+D\left(\vert u_{\ek}^1\vert^2,\vert u_{\ek}^1\vert^2\right)\right] \leq\liminf_{\eps\to0^+}\Eeps^V(u_\eps),  $$
with $m^0+m^1\le M$.
If $m^1=\|u^1\|_{\Ltwo}^2= M-m^0$, then $u_\ek^1(\cdot-x_k^1) \ktends  u^1$ in $L^2$ norm  by the Brezis-Lieb Lemma~\cite{BrezisLieb}, and the proof terminates, with $u^i \equiv 0$ for all $i\ge 2$.

\medskip

\noindent
{\bf Step 5:}   \textit{Iterating the argument.}

If $m^0+m^1<M$, then as in Step~3, the convergence of $\phi^1_\ek(\cdot- x_k^1)\ktends  \phi^1$ is only local and not in the norm of $\Lone$ (and similarly for $u^1_\ek(\cdot- x_k^1)\ktends  u^1$  in  $L^2_{loc}$), and so there is again a remainder part to be separated via Proposition~\ref{CmpctnssBck}.  That is, we may choose radii $\{\rhko\}_{k\in\NN}$ going to infinity and further decompose $\phi^1_\ek(\cdot- x_k^1)$,
$$   \phi^1_\ek(\cdot- x_k^1)\omega_{\rhko}\ktends  \phi^1 \ \text{in $L^1$ norm}, \quad
     \phi^2_\ek:= \phi^1_\ek(\cdot- x_k^1)(1-\omega_\rhko)\ktends  0 \ \text{in $\Loneloc$,}  $$
 with the same consequences as in Step~4:  we first  identify a mass center $x_k^2$ for $\phi_\ek^2$ via Proposition~\ref{FL2.1}.  Since both $\phi^1_\ek, \phi^2_\ek \ktends 0$ in $L^1_{loc}(\RRR)$, we must have both $|x_k^2|,|x_k^2-x_k^1|\ktends\infty$.
Translating and passing to a local $L^1$ limit to find $\phi^2=\frac18 u^2$, we obtain a refined lower bound in terms of $u^0, u^1, u^2$.  Iterating this procedure, after $n$ steps, we have $u^0,\dots u^n\in \XX$ with masses $\|u^i\|_{\Ltwo}^2=m^i$, and translations $\{x^i_k\}_{k\in\NN}$ for each $i=1,\dots,n$, 
such that:

%
%
%

\beq \label{kicker} \left.
\begin{gathered}
u_\ek =  u^0 + \displaystyle\sum_{i=1}^n u^i (\cdot-x^i_k) + u^{n+1}_\ek (\cdot-x^n_k), \ \text{and $u^{n+1}_\ek (\cdot-x^n_k)\ktends  0$ in $\Ltwoloc$;} \\
m^i=\|u^i\|_{\Ltwo}^2, \quad i=0,\dots,n;  \\
|x_k^i|\ktends \infty, \   |  x_k^i- x_k^j|\ktends \infty, \quad  1\le i\neq j;  \\
 M= \sum_{i=0}^n m^i + \lim_{k\to\infty} \|u^{n+1}_\ek\|_{\Ltwo}^2;  \\
\Ez^V(u^0)+\displaystyle\sum_{i=1}^n \Ez^0(u^i)\leq\displaystyle\liminf_{\eps\to0^+}\Ez^V(u_\eps).
\end{gathered}
\right\}
\eeq

If for some $n\in\NN$, the remainder $\phi^i_\ek \to 0$ in the $L^1$ norm, then the iteration terminates at that $n$, and the proof (i) of Theorem~\ref{Th1} is completed by choosing $u^i=0$ for all $i\ge n+1$.  If the iteration continues indefinitely, we must verify that the entire mass corresponding to $\{u_\ek\}_{k\in\NN}$ is exhausted by the $\{u^i\}_{i=0}^\infty$.  It is here that we use $\MM(\{\phi^i_\ek\})$.
When localizing mass in the remainder term $\phi^i_\ek$, the translations  $\{x^i_k\}$ and limit  $\phi^i=\frac18 u^i$ are chosen via Proposition~\ref{FL2.1} in such a way that
$\|\phi^i\|_{\Lone} \ge \frac12 \MM \left(\{ \phi^i_\ek\}\right)$, $i=1,\dots,n$.
In this way, the boundedness of the partial sums $\sum_{i=0}^n m^i \le M$ implies that, should the process continue indefinitely, the residual mass 
$\MM \left(\{ \phi^i_\ek\}\right)\le 2m^i\xrightarrow[i\to\infty]{}0$.
We claim that this implies that 
\beq  \label{allmass}  M=\sum_{i=0}^\infty m^i = \sum_{i=0}^\infty \|u^i\|_{\Ltwo}^2,
\eeq
 and that the entire mass corresponding to $\{u_\ek\}_{k\in\NN}$ is exhausted by the $\{u^i\}_{i=0}^\infty$.  Indeed, if $\sum_{i=0}^\infty m^i=M'<M$, then each remainder sequence has $\|\phi^i_{\ek}\|_{\Lone}\ge {M-M'\over 8}$.  Returning to Step~4, and calculating as in \eqref{L43}, we obtain a lower bound up to a subsequence
$$  \int_{\RRR} |\phi^i_\ek|^{\frac43 } dx  \ge C(M-M'),  $$
for a constant $C$ independent of $k, i$.  Using Lemma~\ref{inqltyBV} we then have a uniform lower bound, 
$$  \MM(\{\phi^i_\ek\}) \ge \sup_{a\in\RRR} \int_{B_1(a)} |\phi^i_\ek| dx   \ge C'(M-M')^3,  $$
for each $i\in\NN$, with $C'$ depending on the upper energy bound $K_0$, but independent of $k,i$.  This contradicts $ \MM(\{\phi^i_\ek\})<2m^i\xrightarrow[i\to\infty]{} 0$.  Hence \eqref{allmass} is established, and passing to the limit $n\to\infty$ in \eqref{kicker} we conclude the proof of (i) of Theorem~\ref{Th1}.
\end{proof}

\section{Upper Bound}

In this section we prove part (ii) of Theorem~\ref{Th1}, the construction of recovery sequences in the $\Gamma$-convergence of $\EepsV $.  As the space $\Hz$ consists of a collection of functions in $\XX$, we build the recovery sequence by superposition of each, using the following lemma:

\begin{Lem}\label{UpprBndLmm}Given $v^0\in \BVStep$ with $\| v^0\|_{\Ltwo}^2=M$, there exists $\eps_0=\eps_0(v^0)>0$ and functions $\{v_\eps\}_{0<\eps<\eps_0}\subset\HM$ of compact support, such that 
\begin{align}
 \| v_\eps-v^0\|_{ L^r(\RRR)}\etends 0, \ \forall 1\leq r<\infty,
\quad \text{ and } \quad \Eeps^V(v_\eps)\etends\Ez^V(v^0).
\end{align}
\end{Lem}
\begin{proof}  The basic construction is familiar, based on that of Sternberg~\cite[Proof of inequalities (1.12) and (1.13)]{Sternberg}, so we highlight the modifications necessary for our case. The first step is to regularize $v^0$.  As compactly supported functions are dense in the $\BV$ norm, we may assume that $\supp v^0$ is bounded.  Next, define a smooth mollifier, using $\varphi\in C^\infty_0(B_1(0))$, $\varphi(x)\ge 0$, $\int_{B_1(0)} \varphi dx  = 1$ to generate $\varphi_n(x)= n^3\varphi (nx) \in C_0^\infty(B_{\frac1n}(0))$. 
Following the proof of regularization of BV functions (see \cite[Theorem 3.42.]{AmbrosioFuscoPallaraBV}), we create a sequence $w_n=\varphi_n*v^0$ which is smooth and supported in a $\frac1n$-neighborhood of the support of $v^0$.  As in \cite{AmbrosioFuscoPallaraBV}, the regularization is obtained as a level surface of $w_n$.  Here, we have two components, corresponding to the regularizations of $v^0_+$ and $v^0_-$, in case $v^0$ takes on both values $\pm 1$. By Sard's Theorem~\cite[3.4.3.]{FedererGMT}, there exist values $t_+\in (0,1)$ and $t_-\in (-1,0)$ for which the boundaries of the sets
$$    F^{+}_n : =\{ x\in\RRR \, |\,  w_n(x)> t_+>0\}, \quad 
     F^{-}_n : =\{ x\in\RRR \, |\,  w_n(x)< t_-<0\} $$
are smooth for each $n\in\NN$, $v^\pm_n:=\One_{F^\pm_n}\ntends  v^0_\pm$ in $\Lone$, and 
$$   \int_{\RRR}|\nabla v^\pm_n|\ntends \int_{\RRR} |\nabla v^0_\pm|.  $$

Note by this construction that the sets $F^\pm_n$ are smooth and disjoint for each $n$.  Hence, the construction in \cite{Sternberg} may be done separately for the components $F^\pm_n$, for any $0<\eps<\eta_n$, with $\eta_n>0$ being chosen so that the neighborhoods of radius $\sqrt{\eps}$ of the boundaries $F^\pm_n$ are disjoint.  Thus, applying the result of Sternberg~\cite{Sternberg}\footnote{We note that the potential in~\cite{Sternberg} has two wells at $u = \pm1$, whereas our transitions connect $v = 0$ to $v = \pm1$, and so our $\tilde{v}_{n,\eps}^{\pm}=\frac{1}{2}(\rho_\eps+1)$ for $\rho_\eps$ as constructed in~\cite{Sternberg}.} for each $n\in\NN$, and each $0<\eps<\eta_n$, there exists $\tilde v^\pm_{n,\eps}(x)\in \Hone$ with $\tilde v^+_{n,\eps}, \tilde v^-_{n,\eps}$ disjointly supported, $0\leq\tilde v^{\pm}_{n,\eps}\leq1$, and 
\beq\label{S1}
    \|\tilde v^\pm_{n,\eps}- v^\pm_n \|_{\Lone}\etends 0, \quad \text{and} \ 
     \int_{\RRR} \left[ {\eps\over 2} |\nabla \tilde v^\pm_{n,\eps}|^2 + 
    {1\over 2\eps} W(\tilde v^\pm_{n,\eps})\right] \etends  \frac18
      \int_{\RRR} |\nabla v^\pm_n| .
\eeq
Writing $\tilde v_{n,\eps}=\tilde v^+_{n,\eps}-\tilde v^-_{n,\eps}$ (again, a disjoint sum for all $0<\eps<\eta_n$), the same properties \eqref{S1} hold for $\tilde v_{n,\eps}$ and $v^0_n = v^+_n - v^-_n$.  

Next, we adjust the $\tilde v_{n,\eps}$ so that for each $n,\eps$,  each has $L^2$ norm equal to $M$, and hence defines a function in $\HM$.  For this we use dilation:  let $\lambda_\eps:=(\| \tilde v_{n,\eps}\|_{\Ltwo}^2/M)^{\frac13}\etends 1$.  We define the rescaled functions $\hat v_{n,\eps}:\RRR\to\R$ by:
$$\hat v_{n,\eps}(x):=\tilde{v}_{n,\eps}\left(\lambda_\eps x\right), \quad \text{and}\quad
\hat v^\pm_{n}(x):= v^\pm_n(\lambda_\eps x).  
$$
First, by rescaling we have $\|\hat v_{n,\eps}\|_{\Ltwo}^2=M$, and so $\hat v_{n,\eps}\in \HM$ for all $n, \eps$.
Next, we observe that, since the supports $F^\pm_n$ of the components of $v^0_n$ are smooth, 
for $|\lambda_\eps - 1|$ sufficiently small, we may estimate
$$  \| \hat v^0_n - v^0_n\|_{\Lone} \le c |\lambda_\eps^{\frac13}-1|\int_{\RRR}|\nabla v^0_n|.  $$
Hence, we have convergence in the $L^1$ norm,
\begin{align}
0\leq\| \hat v_{n,\eps} - v^0_n\|_{\Lone} &\le
     \| \hat v_{n,\eps} - \hat v^0_n \|_{\Lone} + \|\hat v^0_n - v^0_n \|_{\Lone}   \\
 &\le \lambda_\eps^{-1} \| \tilde v_{n,\eps}- v^0_n \|_{\Lone} + c |\lambda_\eps^{\frac13}-1|\int_{\RRR}|\nabla v^0_n|
 \etends  0.
\end{align}
As each of $|\hat v_{n,\eps} |\le 1$ almost everywhere in $\RRR$, and for fixed $n$ each is of uniformly bounded support, the convergence extends to any $L^r(\RRR)$, $r\ge 1$.
Moreover,
\begin{multline}
 \int_{\RRR} \left[ {\eps\over 2} |\nabla \hat v^\pm_{n,\eps}|^2 + 
    {1\over 2\eps} W(\hat v^\pm_{n,\eps})\right]  dx  
 \\ =   \left[\lambda_\eps^{-\frac13}
      \int_{\RRR} {\eps\over 2} |\nabla \tilde v^\pm_{n,\eps}|^2 dx  
        +\lambda_\eps^{-1}\int_{\RRR} {1\over 2\eps} W(\tilde v^\pm_{n,\eps}) dx \right] \etends  \frac18 \int_{\RRR} |\nabla v^0_n| ,
\end{multline}
which holds for each $n\in \NN$.
As in \cite{Sternberg}, by a diagonal argument, there exists $\eps_0=\eps_0(v^0)>0$ so that for any sequence $\ek\ktends  0^+ $ with $\ek<\eps_0$,  we obtain a sequence $\{v_{\ek}\}_{k\in\NN}$ with
\beq\label{S2}
   \|v_{\ek}- v^0 \|_L^r(\RRR) \ktends 0,r\ge1, \quad\text{and} \ 
    \int_{\RRR} \left[ {\ek\over 2} |\nabla v_\ek|^2 + 
    {1\over 2\ek} W(v_\ek)\right] dx   \ktends 
      \frac18\int_{\RRR} |\nabla v^0_\pm|.
\eeq

 The local potential terms also converge by Lemma~\ref{Vlimit}.  Furthermore, by the Hardy-Littlewood-Sobolev inequality \cite[Theorem 4.3]{LiebLoss}  (with $p=6/5=r$),
\begin{align}\label{NnLcCnt}0&\leq\left\vert D(\vert v_\ek\vert^2,\vert v_\ek\vert^2)-D(\vert v^0\vert^2,\vert v^0\vert^2)\right\vert\\
&=\left\vert D(\vert v_\ek\vert^2-\vert v^0\vert^2,\vert v_\ek\vert^2+\vert v^0\vert^2)\right\vert\\
&\leq \bigl\| \vert v_\ek\vert^2-\vert v^0\vert^2\bigr\|_{L^\frac65(\RRR)}\bigl\| \vert v_\ek\vert^2+\vert v^0\vert^2\bigr\|_{L^\frac65(\RRR)}\ktends0.\end{align}
This completes the proof of Lemma~\ref{UpprBndLmm}.
\end{proof}

\begin{proof}[Proof of (ii) of Theorem~\ref{Th1}] 
If $\{u^i\}_{i=0}^{\infty}$ is a finite collection with $N$ nontrivial components, this follows easily from Lemma~\ref{UpprBndLmm}.  Indeed, for any sequence $\ek\ktends  0^+ $ with 
$\displaystyle 0<u_\ek^i < \min_{i=0,\dots,N}\{ \eps_0(u^i)\}_{i=0,\dots,N}$, we
apply the lemma to find $u_\ek^i$  with $u_\ek^i\ktends u^i$, and
$\EkV(u_\ek^i)\ktends \EzV(u^i)$, $i=0,\dots,N$.  We then define the superposition,
$$ u_\ek(x) =u^0_\ek(x)+ \sum_{i=1}^N u^i_\ek(x-x^i_k) , $$
 with translations $\{x^i_k\}_{k\in\NN}$
which will be chosen such that 
dist$\left(\text{supp}\,(u_\ek^i),\text{supp}\, (u_\ek^j)\right)\to\infty$,
 $\forall i\neq j=0,\dots,N$.  We note that this condition on the translations ensures that the energy $\EkV(u_\ek)$ asymptotically splits,
\begin{align} \EkV(u_\ek) &= \EkV(u_\ek^0) + \sum_{i=1}^N \EkV (u_\ek^i) + o(1), \\
   &=\EkV(u_\ek^0) + \sum_{i=1}^N \Ekz (u_\ek^i) + o(1),
\end{align} 
as 
$D(|u_\ek^i(\cdot-x_\ek^i)|^2, |u_\ek^j(\cdot-x_\ek^j)|^2)\ktends 0$, $\forall i\neq j$ (as a consequence of \eqref{pntwsEstmt}), 
and $V(x)\xrightarrow[| x |\rightarrow \infty]{}0$.

If $\Ui$ has an infinite number of nontrivial elements, we must be more careful.  
In particular, as we go down the list of the $\{u^i\}_{i=0}^\infty$, the characteristic length scale of each $u^i$ gets smaller, and for any particular $\eps>0$ there can only be a finite number of $i$ with $0<\eps<\eps_0(u^i)$, for which the trial functions $u^i_\eps$ can be constructed via Lemma~\ref{UpprBndLmm}.

Take any decreasing sequence $\ek\ktends  0^+ $.  By Lemma~\ref{UpprBndLmm} and part (i) of Theorem~\ref{Th1}, for each $i=0,1,2,\dots$ there exist $\eps^i=\eps_0(u^i)>0$
and a sequence $\{u^i_\ek\}_{k\in\NN}$, defiined for $0<\ek<\eps^i$,
 for which
$$  \|u^i_\ek\|_{\Ltwo}^2= m^i, \quad \EkV(u^i_\ek)\ktends \EzV(u^i), \quad 
\| u^i_\ek-u^i\|_{\Ltwo}\ktends 0.
$$
By taking $\eps^i$ smaller if necessary, we may also assume:
\beq\label{E1} \left.
\begin{gathered}
 \bigl| 
\EkV(u^0_\ek)- \EzV(u^0)\bigr| < {\EzV(u^0)\over 10} \ \text{and} \ 
   \|u^0_\ek - u^0\|_{\Ltwo}^2 <{m^0\over 10}, \quad  \, 0<\ek<\eps^0, \\
     \bigl|\Ekz(u^i_\ek)- \Ezz(u^i)\bigr| < {\Ezz(u^i)\over 10}, \ \text{and} \ 
     \| u^i_\ek - u^i \|_{\Ltwo}^2 < {m^i\over 10}, \quad  \, 0<\ek<\eps^i, \  i=1,2,3,\dots
     \end{gathered}\right\}
\eeq
Again taking $\eps^i$ smaller if necessary, we may assume $0<\eps^i<\eps^{i-1}$.
We now construct $U_{\ek}$ as follows:  for each $k\in\NN$, choose the largest integer $n_k\ge 0$ such that $0<\ek<\eps^i$ for all $i\le n_k$. Note that $n_k\ktends\infty$.  We recall that the $u^i_\ek$ are all compactly supported,  and define $R^i_\ek$ by supp$\, ( u^i_\ek)\subset B_{R_\ek^i}(0)$.  Let 
$$ \bar{R}_\ek=\max_{i=1,\dots,n_k} R^i_\ek.  $$
  Then, we choose the translations $x_k^i\in\RRR$, $i=1,\dots,n_k$, so that
$$\vert x_k^i-x_k^j\vert>\max\left\{ 4\bar{R}_\ek, 2^k\right\} \ktends\infty.$$

Then, we set 
\beq\label{E3}
U_{\ek}(x)  : = u^0_\ek (x) + \sum_{i=1}^{n_k}   u^i_\ek(x-x^i_k),
\eeq
 which is a disjoint sum.  As $V\geq0$, we have 
\beq
\EkV(U_{\ek})
   \leq\EkV(u^0_\ek) + \sum_{i=1}^{n_k} \Ekz(u^i_\ek)
   +\sum_{i,j=1\atop i\neq j}^{n_k}D(|u^i_\ek(\cdot-x_\ek^i)|^2,|u^j_\ek(\cdot-x_\ek^j)|^2).
\eeq

  We claim that the last term on the right side above is negligible.  Indeed,
 for $x\in B_{\bar R_\ek}(x^i_\ek)$ and $x\in B_{\bar R_\ek}(x^j_\ek)$, we have the pointwise estimate,
 \beq\label{pntwsEstmt}
 \left|  {1\over |x-y|} - {1\over |x_\ek^i - x_\ek^j|}\right| 
   \le {2\bar R_\ek \over \left(|x_\ek^i - x_\ek^j| - \bar R_\ek\right)^2}
     \le {4  \bar R_\ek \over  |x_\ek^i - x_\ek^j|^2}
     \le {1 \over  |x_\ek^i - x_\ek^j|}.
 \eeq
 Hence, 
\begin{align}
\sum_{i,j=1\atop i\neq j}^{n_k}
D(|u^i_\ek(\cdot-x_\ek^i)|^2,|u^j_\ek(\cdot-x_\ek^j)|^2) 
&\le  \sum_{i,j=1\atop i\neq j}^{n_k}
     { \|u^i_\ek\|_{\Ltwo}^2 \|u^j_\ek\|_{\Ltwo}^2 \over  |x_\ek^i - x_\ek^j|} \\
     &\le 2^{-k}\sum_{i,j=1\atop i\neq j}^{n_k}
     m^i\, m^j \le  2^{-k}M^2 \ktends 0.
\end{align}
As a result, for any given $\delta>0$, there exists $K\in\NN$ for which
\begin{equation}\label{E4}
\EkV(U_{\ek})\leq\EkV(u^0_\ek) + \sum_{i=1}^{n_k} \Ekz(u^i_\ek)+ {\delta\over 5},
\quad \forall k\ge K.
\end{equation}
Note also that the mass 
$\|U_\ek\|_{\Ltwo}^2 = \sum_{i=0}^{n_k} m^i=:M_k< M$, but it will approach $M$ as $n_k\to\infty$ and components are successively added to the sum.

We next show that 
\beq\label{E90} \limsup_{k\to\infty}\EkV(U_\ek)\leq\FzV(\Ui).
\eeq
 In case $\FzV(\{u^i\}_{i=0}^\infty)=\infty$, (which is possible because the nonlocal terms are not necessarily summable for all $\{u^i\}_{i=0}^\infty\in \Hz$),
there is nothing to prove.  When $\FzV(\{u^i\}_{i=0}^\infty)<\infty$,   choose $N\in\NN$ (which is independent of $k$), for which
\beq\label{E5}
\sum_{i=N+1}^\infty m^i < \delta \ \text{and} \ 
\sum_{i=N+1}^\infty \Ezz(u^i) < {\delta\over 5}.
\eeq
From Lemma~\ref{UpprBndLmm}, taking $K\in\NN$ larger if necessary, we have for all $k\ge K$,
\beq\label{E6}
\sum_{i=0}^N \| u_\ek^i - u^i\|_{\Ltwo}^2+
  \bigl| \EkV(u^0_\ek)-\EzV(u^0)\bigr|+ 
     \sum_{i=1}^N \left|\Ekz(u^i_\ek) - \Ezz (u^i)\right| < {\delta\over 5}.
\eeq
Using \eqref{E4}, \eqref{E6}, \eqref{E5}, and \eqref{E1}, we estimate
\begin{align}\label{E7}
\EkV(U_\ek) - \FzV (\Ui)&\leq\EkV(U_\ek) - \EzV(u^0)-\sum_{i=0}^N \Ezz(u^i)\\
    &\le \EkV(u^0_\ek) - \EzV(u^0) 
     + \sum_{i=1}^N[\Ekz(u^i_\ek)-\Ezz(u^i)] \\
     &\ \ \ + \sum_{i=N+1}^{n_k} |\Ekz(u^i_\ek) - \Ezz(u^i)|
         + \sum_{i=N+1}^{n_k} \Ezz(u^i) +{\delta\over 5}
     \\
    &< {2\delta\over 5} + \frac{11}{10} \sum_{i=N+1}^{n_k}\Ezz(u^i)
    < \delta,
 \end{align}
for all $k\ge K$.  Hence \eqref{E90} is verified.

We next prove that 
\beq\label{E91}
\left\vert\left\vert U_{\ek}-\left(u^0+ \sum_{i=0}^\infty u^i(x-x_k^i)\right)\right\vert\right\vert_{\Ltwo}\ktends0.
\eeq
 For given $\delta>0$, let $N, K\in\NN$ be  as in \eqref{E5} and \eqref{E6}.
Then, for all $k\ge K$, using \eqref{E5}, \eqref{E1}, and \eqref{E6}, we estimate
\begin{align}\label{E9}
\left\|  U_\ek - \left(u^0 + \sum_{i=0}^\infty u^i(x-x_k^i)\right)\right\|_{\Ltwo}
  &\le  \sum_{i=0}^N \| u_\ek^i - u^i\|_{\Ltwo}
   + \sum_{i=N+1}^{n_k} \|u_\ek^i - u^i \|_{\Ltwo}\\
   &\ \ \  + \sum_{i=n_k+1}^\infty \|u^i\|_{\Ltwo} \\
  &\le  {\delta\over 5} + \sum_{i=N+1}^{n_k} {m^i\over 10} + {\delta\over 5} <\delta,
\end{align}

It remains to correct the mass of $U_\ek$, so that each $\|U_\ek\|_{\Ltwo}^2=M$.  This is done as in Lemma~\ref{UpprBndLmm}, dilating each component $u^i_\ek$ by the scaling factor
$\lambda_k=(M_k/M)^{\frac13} \ktends  1$, that is, by setting 
$$ u_\ek(x)=  u^0_\ek ( \lambda_k x) + \sum_{i=1}^{n_k} u^i_\ek (\lambda_k(x-x_k)).  $$
Then $\|u_\ek\|_{\Ltwo}^2 = M$,  $k\in\NN$, $\|u_\ek - U_\ek\|_{\Ltwo}\ktends  0$, and $\vert\EkV(u_\ek)-\EkV(U_\ek)\vert\ktends0$, since $\lambda_k\ktends 1$.
 This concludes the proof of Theorem \ref{Th1}.
\end{proof}

\section{Minimizers}

In this section we examine the connection between minimizers of the liquid drop and TDFW functionals.  The compactness of minimizing sequences being a delicate issue which is shared by the two models.

First, whether the minimum in $\eepsV(M)$ is attained or not, the infimum values converge as $\eps\to 0^+ $:

\begin{Lem}\label{lem:infs}
Assume $V$ satisfies \eqref{HypV}.  Then, for all $M>0$,
$\displaystyle  \eepsV(M) {\xrightarrow[\eps \rightarrow 0^+]{}}\ezV(M).$
\end{Lem}

\begin{proof}   The proof is standard. 
 Take any sequence $\eps_n\ntends 0^+$. Then, $\forall n$, $\exists u_{\eps_n}\in\HM$  with $\|u_{\eps_n}\|_{\Ltwo}^2=M$ and $\EepsnV (u_{\eps_n})\le \eepsnV(M)+\eps_n$.
Using $u^0=\One_\BRM$ in Lemma~\ref{UpprBndLmm}, we may conclude that $\{\eepsnV(M)\}_{n\in\NN}$ is bounded, and so by 
 Theorem~\ref{Th1}~(i), $\exists \{u^i\}_{i=0}^\infty \in\Hz$ and a subsequence (not relabelled)  $\eps_n\ntends 0^+ $ with 
$$ \ezV (M)\le \FzV (\{u^i\}_{i=0}^\infty )\le \liminf_{n\to\infty} \EepsnV (u_{\eps_n}) = \liminf_{n\to\infty} e_{\eps_n}^V(M).  $$
For the complementary inequality, for any $\delta>0$, $\exists \{v^i\}_{i=0}^\infty \in\Hz$ with 
$\FzV (\{v^i\}_{i=0}^\infty )< \ezV (M)+\delta$.  Then, by (ii) in Theorem~\ref{Th1}, for any $n\in\NN$, $\exists v_n\in\HM$ with 
$$   \ezV (M)+\delta> \FzV (\{v^i\}_{i=0}^\infty ) \ge \limsup_{n\to\infty } \EepsnV (v_n) 
   \ge \limsup_{n\to\infty } \eepsnV(M).  $$
Putting the above inequalities together, and letting $\delta\to 0^+ $, every sequence $\eps_n\to 0$ contains a subsequences for which $\eepsnV(M)\ntends \ezV(M)$.  As the limit is unique, the lemma follows.\end{proof}

\begin{proof}[Proof of Corollary \ref{cor:super}]
In \cite[Theorems $1$ and $2$]{AlamaBronsardChoksiTopalogluLongRange} it is proven that for $V$ satisfying \eqref{Long},
the minimum for both $\Ez^V$ and $\Eeps^V$ are attained, correspondingly.  Indeed, the proof of these results in \cite{AlamaBronsardChoksiTopalogluLongRange} actually yields the stronger conclusion that \underbar{all} minimizing sequences for either the TDFW or liquid drop functionals are convergent.  Thus, $\forall \eps>0$, $\exists u_\eps\in\HM$ which attains the minimum, $\eepsV(M)=\EepsV (u_\eps)$.  By Lemma~\ref{lem:infs}, $\EepsV (u_\eps)\etends \ezV (M)$, so for any sequence $\eps_n\ntends 0^+ $, 
by Theorem~\ref{Th1}~(i), $\exists \{ u^i\}_{i=0}^\infty\in \Hz$ with 
$$ \FzV (\{ u^i\}_{i=0}^\infty)\le \liminf_{n\to\infty} \mathcal{E}_{\eps_n}^V (u_{\eps_n}) = \ezV (M).  $$
Defining $m^i=:\|u^i\|_{\Ltwo}^2$, we have
\begin{equation}\label{corcontra}  \ezV (M) = \ezV (m^0) + \sum_{i=1}^\infty  \ezz (m^i),  
\end{equation}

 We now claim that $u^i=0$, $\forall i\ge 1$, in which case $u_{\eps_n}\ntends u_0$ in $\Ltwo$, as desired.  Indeed, assume the contrary, $m^1>0$.
We then obtain a contradiction by using Step~6 in the proof of Theorem~1 of \cite{AlamaBronsardChoksiTopalogluLongRange}.  Indeed, by choosing compactly supported $v^0, v^1\in \HM$ whose energies are close to the infima $\ezV (m^0), \ezz(m^1)$ as in Step 6, we obtain the strict subadditivity condition,
$$  \ezV (M) < \ezV (m^0) + \ezz(m^1) + \ezz(M-m^0-m^1) \le \ezV (m^0) + \sum_{i=1}^\infty  \ezz (m^i), 
$$
which contradicts \eqref{corcontra}.
\end{proof}

Analyzing the possible loss of compactness in minimizing sequences for $e^Z_\eps(M)$, $\eps\ge 0$ and $Z\ge 0$, requires the use of concentration-compactness methods~\cite{LionsConcentrationPart1}.  The following are standard results for problems where loss of compactness entails splitting of mass to infinity:
\begin{Lem}\label{conc}  Assume $V$ satisfies \eqref{HypV}.  Then,
for any $\eps\ge 0$ and $M>0$,
\begin{enumerate}
\item[(i)]  If $\forall m^0\in (0,M)$, 
\begin{equation}\label{strictsub} \eepsV(M) < \eepsV(m^0) + e_\eps^0(M-m^0), 
\end{equation}
then all minimizing sequences for $\eepsV(M)$ are precompact  in $\Ltwo$.
\item[(ii)]  If there exist  divergent minimizing sequences for $\eepsV(M)$, then $\exists m^0\in (0,M)$ such that $\eepsV(m^0)$ attains a minimizer and $\eepsV(M) = \eepsV(m^0) + e_\eps^0(M-m^0)$.
\end{enumerate}
\end{Lem}

Statement (ii) is a  slight strengthening of the contrapositive of (i).  The proof for the TFDW functional was done in \cite[Corollary II.2 part (ii)]{Lions}, and for liquid drop models it may be derived from \cite[\color{blue}Lemma 6]{AlamaBronsardChoksiTopalogluDroplet}; although it is stated there for $V$ of a special form, in fact it is true for a much larger class including those satisfying \eqref{HypV}.

Next, we specialize to the atomic case, 
$$    V(x) = {Z\over |x|},  $$
and present the following refinement of the existence result of \cite{LuOttoLiquidDropBackground} for the liquid drop model with atomic potential:

\begin{Prop}\label{prop:LO}
There exists a constant $\mu_0>0$ such that for all $Z\ge 0$ and for all $M\in (0, Z+\mu_0)$:
\begin{enumerate}
\item[(i)] All minimizing sequences for $e^Z_0(M)$ are precompact.
\item[(ii)] The unique minimizer (up to translations if $Z=0$) of $e^Z_0(M)$ is the ball $\BRM(0)$ of radius $r_M=\left({3M\over 4\pi}\right)^{1/3}$.
\end{enumerate}
\end{Prop} 

\begin{proof}
Statement (ii) is proven in Theorem~2 of \cite{LuOttoLiquidDropBackground}, using Theorem~2.1 in \cite{Julin}.  (The special case $Z=0$ was proven earlier in \cite{KnupferMuratov}.)  We sketch the proof of (i), since we will need certain definitions and estimates for (ii).  As in Julin \cite{Julin}, we define an asymmetry function corresponding to a fixed set $\Om$ of finite perimeter,
$$   \gamma(\Om):= \min_{y\in\RRR} \int_{\RRR} {\mathbbm{1}_{B}(x)-\mathbbm{1}_\Om(x+y)\over |x|} dx,  $$
where $B=\BRM(0)$ is the ball of mass $M$ centered at the origin.  The quantitative isoperimetric inequality (see (2.3) of \cite{Julin} or \cite{FuscoJulin}) then asserts the existence of a universal constant $\mu_0>0$, such that
\begin{equation}\label{QI}
\int_{\RRR} |\nabla\mathbbm{1}_\Om|-\int_{\RRR} |\nabla\mathbbm{1}_B| \ge\mu_0 \gamma(\Om),
\end{equation}
with equality if and only if $\Om$ is a translate of $B$.
Then, as in the proof of Theorem~1.1 of \cite{Julin} in the three-dimensional case, we may estimate the difference in the nonlocal terms by the asymmetry,
$$   D(\mathbbm{1}_B,\mathbbm{1}_B)-D(\mathbbm{1}_\Om,\mathbbm{1}_\Om) \le |B|\gamma(\Om).  $$
The optimality of the ball $B=\BRM$ follows easily from this:  assume $\Om$ is of finite perimeter, with $|\Om|=M$.  Then, provided $\Om$ is not a translate of the ball $B=\BRM$,
\begin{multline}\label{ballbest}   \EzZ(\mathbbm{1}_\Om)-\EzZ(\mathbbm{1}_B) > (\mu_0-M) \gamma(\Om) 
   + Z\left( \int_{\RRR} {\mathbbm{1}_B(x)-\mathbbm{1}_\Om(x)\over |x|} dx\right) \\ \ge \left(Z+\mu_0 - M\right)\gamma(\Om)>0,  
\end{multline}
for all $M<Z+\mu_0$.

To obtain (i), the precompactness of all minimizing sequences, we use the above to establish strict subadditivity of $\ezZ(M)$, as in Lions~\cite{LionsConcentrationPart1}.  Let $M=m^0+m^1$ with $m^0,m^1>0$; we will show that \eqref{strictsub} holds, and then by Lemma~\ref{conc} all minimizing sequences for $\ezZ(M)$ are precompact.


Since $0<m^0<M<Z+\mu_0$, both $\ezZ(M)$, $\ezZ(m^0)$ are attained by balls $B=\BRM(0)$, $B^0=B_{r_{m^0}}(0)$.  For any $\delta>0$ (to be chosen later), we may choose a bounded open set $\omega$ with $0\in\omega$, $|\omega|=m^1$, and $\Ez(\mathbbm{1}_\omega)<\ezz(m^1)+\delta$.  Note that if $m^0\ge Z$, then $0<m^1<\mu_0$ and we may choose $\omega=B^1=B_{r_{m^1}}$ which attains $\ezz(m^1)$.

Define $\omega_\xi:=\omega+\xi$, and $\Om=\Om_\xi = B^0\cup \omega_\xi$, with $|\xi|$ sufficiently large that the union is disjoint.  We first claim that $\exists R>1$ such that $\gamma(\Om_\xi)\ge C>0$ is bounded away from zero for all $\xi$ with $|\xi|>R$, with constant $C=C(m^0,m^1)$.  
 Indeed, for $y\in\RRR$ define
$$  v=v^0+v^1, \quad v^0(y)= \int_{B^0} {dx\over |x-y|}, \quad
    v^1(y) = \int_{\omega_\xi} {dx\over |x-y|}, $$
so that 
$$   \gamma(\Om_\xi) = \int_{B} {dx\over |x|} - \max_{y\in\RRR} v(y).  $$
Hence, to bound $\gamma(\Om_\xi)$ from below we must bound $v(y)$ uniformly from above.
As  $-\Delta v = 4\pi(\mathbbm{1}_{B^0}(y) + \mathbbm{1}_{\omega_\xi}(y))$ in $\RRR$, it attains its maximum at  $y\in \overline{\Om_\xi}=\overline{B^0}\cup \overline{\omega_\xi}$.  Thus, there are two possibilities:  if the maximum occurs at  $y\in \overline{B^0}$, then $v(y)=v^0(y)+ O(|\xi|^{-1})$.  Since $v^0$ is maximized at $y=0$, there exists $C_0=C_0(M,m^0)$ and $R>1$ with  
$$   \gamma(\Om_\xi) \ge \int_{B\setminus B^0} {dx\over |x|} - O(|\xi|^{-1}) \ge C_0>0,  $$
for all $|\xi|>R$.  

In case the maximum of $v$ occurs at $y\in \overline{\omega_\xi}$, then 
$v(y)=v^1(y) + O(|\xi|^{-1})$.  For any domain $D$ with $|D|=m^1$ we have
$$   \int_D {dx\over |x|} \le \int_{B^1} {dx\over |x|}, $$
where $B^1=B_{r_{m^1}}(0)$ is the ball with mass $m^1$.
It follows that
$$  v^1(y) = \int_{\omega_\xi} {dx\over |x-y|}  \le \int_{B^1}{dx\over |x|}.  $$
Therefore, as in the previous case, there exist $C_1=C_1(M, m^1)$ and $R>1$ with
$\gamma(\Om_\xi) \ge C_1>0$,  for all $|\xi|>R$, and the claim is established, with $C=\min\{C_0,C_1\}$.

To conclude, we choose $0<\delta<\frac12 \left(Z+\mu_0 - M\right) C\le \frac12 \left(Z+\mu_0 - M\right)\gamma(\Om_\xi)$, for any $|\xi|>R$, and using \eqref{ballbest},
\begin{align}
\ezZ(M) = \EzZ( \mathbbm{1}_{B}) &<  \EzZ(\mathbbm{1}_{\Om_\xi}) - \left(Z+\mu_0 - M\right)\gamma(\Om_\xi) \\
& \le \EzZ(\mathbbm{1}_{B^0}) + \EzZ(\mathbbm{1}_{\omega_\xi}) - \left(Z+\mu_0 - M\right)\gamma(\Om_\xi) + 
 2\int_{B^0}\int_{\omega_\xi} {dx\, dy \over |x-y|}  \\
&\le  \EzZ(\mathbbm{1}_{B^0}) + \mathcal{E}_0^0(\mathbbm{1}_{\omega}) - \left(Z+\mu_0 - M\right)\gamma(\Om_\xi) +O(|\xi|^{-1}) \\
&\le \ezZ(m^0) + \ezz(m^1) +\delta - \left(Z+\mu_0 - M\right)\gamma(\Om_\xi) +O(|\xi|^{-1}).
\end{align}

Taking $|\xi|$ sufficiently large, \eqref{strictsub} holds for all $M\in (0, Z+\mu_0)$.
\end{proof}

\begin{Rmk}\label{uniquegm}  Thanks to Proposition~\ref{prop:LO}, we may conclude that for the liquid drop model with $V(x)=Z/|x|$ with $0<M<Z+\mu_0$, the {\em unique} generalized minimizer (see Definition~\ref{genmin})  is the singleton $\{u^0=\mathbbm{1}_{\BRM}\}$.  Indeed, this will be true for any functional which satisfies the strict subadditivity condition \eqref{strictsub}.
\end{Rmk}

Next, we prove Theorem~\ref{minlim}.  In fact, we prove the following slightly more general version, which will also be a step towards the proof of Theorem~\ref{almostminlim}.

\begin{Lem}\label{minlem}
Let $M>0$ and $\delta_n,\eps_n\ntends 0$.  Assume $u_n\in\HM$ for which $\EepsnV(u_n)\le \eepsnV(M)+\delta_n$ for each $n\in\NN$.  Then, there exists a subsequence and a generalized minimizer 
$\{u^0, \dots, u^N\}$ of $\EzV $ for which \eqref{Cmpctnss} and \eqref{xfar} hold for $i=0,\dots,N$, and 
$$ \FzV (\{u^i\}_{i=0}^N)=  \ezV(M) = \lim_{n\to\infty}  \eepsnV(M).  $$
\end{Lem}

\begin{proof}
By (i) of Theorem~\ref{Th1}, there exists a subsequence along which $u_n$ decomposes as in \eqref{Cmpctnss}, with $\{u^i\}_{i=0}^\infty\in\Hz$ satisfying \eqref{LwrBnd}.  By (ii) of Theorem~\ref{Th1} the upper bound construction provides sequences $v_{\eps_n}\in\HM$ yielding the opposite inequality, 
$$\FzV (\{u^i\}_{i=0}^\infty)\ge\limsup_{n\to\infty}\EepsnV(v_n)\ge \lim_{n\to\infty} \eepsnV(M). 
$$
  Hence, by Lemma \ref{lem:infs} we have
$$  \FzV (\{u^i\}_{i=0}^\infty) = \lim_{n\to\infty} \EepsnV(u_n) = \lim_{n\to\infty}  \eepsnV(M) = \ezV(M).  $$
Let $m^i=\|u^i\|_{\Ltwo}^2$.  
It suffices to show that $u^0$ minimizes $\ezV(m^0)$ and $u^i$ minimizes $\ezz(m^i)$, for each $i\geq1$, and that all but a finite number of the $u^i\equiv 0$.  First, by \eqref{GmmLmmt} we have
\begin{multline}   \ezV(m^0)+\sum_{i=1}^\infty \ezz(m^i)\\ \le 
     \EzV(u^0)+ \sum_{i=1}^\infty \Ez^0(u^i) = \FzV (\{u^i\}_{i=0}^\infty) =  \ezV(M) 
     \le  \ezV(m^0)+\sum_{i=1}^\infty \ezz(m^i),  
\end{multline}
the last step by the Binding Inequality (subadditivity) of $e_0$ see e.g. \cite{AlamaBronsardChoksiTopalogluDroplet}.)
As each term is non-negative, equality holds in each relation. Furthermore, as $ \ezV(m^0)\le  \EzV (u^0)$ and each $\ezz(m^i)\le \Ez^0(u^i)$, we must have equality in these as well.  This proves that each $u^i$, $i\ge 0$, is minimizing.

Finally, suppose infinitely many $u^i\not\equiv 0$.  Then, by the convergence of the series, $0<m^i<\mu_0$ for all but finitely many $i$;  assume $0<m^j, m^{j+1}<\mu_0$.  Then,   as in the proof of Proposition~\ref{prop:LO},
we obtain the strict subadditivity condition,  $\ezz(m^j) + \ezz(m^{j+1})> \ezz(m^j+m^{j+1})$.   But then
$$
 \ezV(M) = \ezV(m^0)+ \sum_{i=1}^\infty \ezz(m^i) >  \ezV(m^0)+ \sum_{i\neq j,j+1} \ezz(m^i) + \ezz(m^j+m^{j+1}) \ge  \ezV(M),
$$
a contradiction.
\end{proof}

We finish with the proof of Theorem~\ref{almostminlim}.  

\begin{proof} 
 Recall that we assume $V(x)=Z/|x|$, $Z>0$.  For (a), $0<M\le Z$, the (relative) compactness of all minimizing sequences for $e^Z_\eps(M)$ was proven by Lions~\cite[\color{blue}Corollary II.2.]{Lions}.  
  Take any sequence $\eps_n\to 0$ and let $u_n\in\HM$ with $\EepsnZ(u_n)=\eepsnZ(M)$.  By Lemma~\ref{minlem}, there exists a generalized minimizer of $\ezZ(M)$, $\{u^i\}_{i=0}^N$, such that \eqref{Cmpctnss} and \eqref{xfar} hold for $i=0,\dots,N$,  and a subsequence, for which
$$ \FzZ (\{u^i\}_{i=0}^N)=  \ezZ(M)= \lim_{n\to\infty}  \eepsnZ(M).  $$
By Remark~\ref{uniquegm}, $N=0$ and  $u_n\ntends u^0$ in $\Ltwo$, which attains the minimum in $e^Z_0(M)$.  As $u^0=\mathbbm{1}_{\BRM}$ is unique, the limit exists for any sequence $\eps\to 0$.

For (b), first note that if there is a sequence $\eps_n\ntends 0^+ $ for which $e_{\eps_n}^Z(M)$ attains its minimum at $u_n\in \HM$, then by the same argument as for (a) we obtain the conclusion of the Theorem with $M_{\eps_n}=M$.  It therefore suffices to consider sequences $\eps_n\ntends 0^+ $ for which the minimum in $e_{\eps_n}^Z(M)$ is not attained.  By part (ii) of Lemma~\ref{conc}, for each $n$ there exists $m^0_n\in (0,M)$ such that  
$$e^Z_{\eps_n}(M) = e^Z_{\eps_n}(m^0_n) + e^0_{\eps_n}(M-m^0_n),  $$
and there exists $u_n\in H^1(\RRR)$ with $\|u_n\|_{\Ltwo}^2=m^0_n$ and  $\EepsnZ(u_n)= e^Z_{\eps_n}(m^0_n)$.  For each $n$, we may choose $v_n\in H^1(\RRR)$ with compact support and $\|v_n\|_{\Ltwo}^2=M-m^0_n$ and for which $\mathcal{E}^0_{\eps_n}(v_n)< e^0_{\eps_n}(M-m^0_n) +\eps_n$.  Next, choose radii $\rho_n$ in the smooth cut-off $\omega_{\rho_n}$ defined in \eqref{OmgEps}, such that $\tilde u_n=u_n\omega_{\rho_n}$ satisfies both
$\|\tilde u_n- u_n\|_{\Ltwo}^2\ntends 0$ and $\vert\EepsnZ(\tilde u_n)-\EepsnZ(u_n)\vert\ntends 0$.  We also choose $\xi_n\in\RRR$ such that $\tilde u_n$ and $v_n(\cdot+\xi_n)$ have disjoint supports for each $n$, and $|\xi_n|\ntends \infty$.  Set $U_n(x)=\tilde u_n(x) + v_n(\cdot + \xi_n)$, so that
\beq \label{Un}   \|U_n\|_{\Ltwo}^2 = \|\tilde u_n\|_{\Ltwo}^2 + \|v_n\|_{\Ltwo}^2 \ntends M, \quad\text{and}\quad
        \vert\EepsnZ(U_n) - \eepsnZ(M)\vert \ntends 0.  
\eeq        
By Lemma~\ref{lem:infs}, $\EepsnZ(U_n) \ntends e^Z_0(M)$, so applying (i) of Theorem~\ref{Th1} 
there exists $\{u^i\}_{i=0}^\infty\in \Hz$ for which \eqref{Cmpctnss} and \eqref{xfar} hold, and 
$$ \ezZ(M)\le\FzZ (\{u^i\}_{i=0}^\infty)\le \liminf_{n\to\infty} \EepsnZ(U_n) = \ezZ(M).  $$
Thus, $\FzZ (\{u^i\}_{i=0}^\infty)= \ezZ(M)$.  By Remark~\ref{uniquegm}, $u^i\equiv 0$ for all $i\ge 1$ and $u^0=\mathbbm{1}_{\BRM}$ minimizes $\ezV (M)$.  From \eqref{Cmpctnss} we conclude that $U_n=\tilde u_n + v_n (\cdot+\xi_n)\ntends u^0$ in $L^2(\RRR)$.  Since for every fixed compact set $K\subset\RRR$ we have $U_n =u_n $ almost everywhere in $K$ and for all sufficiently large $n$, it follows that $u_n\ntends u^0$ in $\Ltwoloc $ and pointwise almost everywhere up to a subsequence.
Fix the compact set $K$ with $\BRM\Subset K$.  Then,
$$ M = \int_K |u^0|^2 \le \liminf_{n\to\infty} \int_K |u_n|^2\, dx \le  \liminf_{n\to\infty} m_n^0 \le M.  $$
Each of the above quantities is therefore equal, and $\lim_{n\to\infty} m_n^0=\lim_{n\to\infty} \|u_n\|_{\Ltwo}^2=M$.
  Consequently, we have both $u_n\ntends u^0$ and  $v_n\ntends 0$ globally in $\Ltwo$.  In conclusion, taking $M_{\eps_n}=m_n^0$, $\eepsnZ(M_{\eps_n}=m_n^0)$ is attained at $u_{\eps_n}=u_n$, $M_{\eps_n}\ntends M$, and $u_n\ntends u^0=\mathbbm{1}_{\BRM}$ in $\Ltwo $.
\end{proof}

\vspace{1cm}
{\begin{center}\textbf{Acknowledgments}\end{center}}
The authors were supported by NSERC (Canada) Discovery Grants.


\begin{thebibliography}{99}
\bibitem{AAB}  L. Aguirre Salazar, S. Alama, and L. Bronsard,
``Mass splitting in the Thomas-Fermi-Dirac-von Weizs\"{a}cker model with background potential.''
Journal of Mathematical Physics $\mathbf{61}$, $021502$ $(2020)$.
\bibitem{AlamaBronsardChoksiTopalogluDroplet} S. Alama, L. Bronsard, R. Choksi, and I. Topaloglu, ``Droplet breakup in the liquid drop model with background potential," Commun. Contemp. Math.  August $1850022$ $(2018)$.
\bibitem{AlamaBronsardChoksiTopalogluLongRange} S. Alama, L. Bronsard, R. Choksi, and I. Topaloglu, ``Ground-states for the liquid drop and TFDW models with long-range attraction," J. Math. Phys. $\mathbf{58}$, $103503$ $(2017)$.
\bibitem{ABLW}  S. Alama, L. Bronsard, X. Lu, and C. Wang, ``Periodic Minimizers of a Ternary Non-Local Isoperimetric Problem.''  \textit{preprint at arXiv:1912.08971}.  To appear in { Indiana U. Math. Jour.}
\bibitem{AmbrosioFuscoPallaraBV} L. Ambrosio, N. Fusco, and D. Pallara, \underline{Functions of Bounded Variation and Free Discontinuity} \underline{Problems}, Oxford Mathematical Monographs, Oxford University Press, New York $(2000)$.
\bibitem{BenguriaBrezisLieb} R. Benguria, H. Br\'ezis, and E.H. Lieb, ``The Thomas-Fermi-von Weizs\"acker theory of atoms and molecules," Commun. Math. Phys. $\mathbf{79}(2)$, $167$-$180$ $(1981)$.
\bibitem{BonaciniCristoferi}M. Bonacini and R. Cristoferi, ``Local and global minimality results for a nonlocal isoperimetric problem on $\R^N$,'' SIAM J. Math. Anal. $\mathbf{46}(4)$ $2310$-$2349$ $(2014)$.
\bibitem{Braides} Braides, Andrea. "A handbook of $\Gamma$-convergence." Handbook of Differential Equations: stationary partial differential equations. Vol. 3. North-Holland, 2006. 101-213.
\bibitem{BrezisLieb} H. Brezis and E.H. Lieb, ``A relation between pointwise convergence of functions and convergence of functionals,'' Proc. Amer. Math. Soc. $\mathbf{88}$, $486$-$490$ $(1983)$.
\bibitem{ChoksiMuratovTopaloglu} R. Choksi, C.B. Muratov, and I. Topaloglu, ``An Old Problem Resurfaces Nonlocally: Gamow's Liquid Drops Inspire Today’s Research and Applications,'' Notices of the AMS $\mathbf{64}(11)$, $1275$-$1283$ $(2017)$.	
\bibitem{CP} R. Choksi and M.A. Peletier, ``Small volume fraction limit of the diblock copolymer problem: II. Diffuse-interface functional.'' SIAM J. Math. Anal.  $\mathbf{43}(2)$, $739$-$763$ $(2011)$.
\bibitem{FedererGMT} H. Federer. \underbar{Geometric Measure Theory}, Springer $(1996)$.
\bibitem{FigalliFuscoMaggiMillotMorini} A. Figalli, N. Fusco, F. Maggi, V. Millot, and M. Morini, ``Isoperimetry and stability properties of balls with respect to nonlocal energies,'' Comm. Math. Phys. $\mathbf{336}(1)$ $441$-$507$ $(2015)$.
\bibitem{FrankLieb} R.L. Frank and E.H. Lieb, ``A compactness lemma and its applications to the existence of minimizers for the liquid drop model,'' SIAM J. Math. Anal. $\mathbf{47}(6)$, $4436$-$4450$ $(2015)$.
\bibitem{FrankNamVBosch} R.L. Frank, P.T. Nam, and H. Van Den Bosch, ``The Ionization Conjecture in Thomas-Fermi-Dirac-von Weizs\"{a}cker Theory,'' Comm. Pure Appl. Math. $\mathbf{71}$, $577$-$614$ $(2018)$. 
\bibitem{FuscoJulin} N. Fusco and V.  Julin, ``A strong form of the quantitative isoperimetric inequality,'' Calc. Var. Partial Differential Equations  $\mathbf{50}(3$-$4)$, $925$-$937$ $(2014)$.
\bibitem{JeSt} R.L. Jerrard and P.  Sternberg, ``Critical points via $\Gamma$-convergence: general theory and applications.''
 J. Eur. Math. Soc. (JEMS)  {\bf 11}  (2009),  no. 4, 705--753.
\bibitem{Julin} V. Julin, ``Isoperimetric problem with a Coulombic repulsive term,'' Indiana U. Math. Jour., $\mathbf{63}$, $77$-$89$ $(2014)$.
\bibitem{KnupferMuratov} H. Kn\"{u}pfer and C. B. Muratov, ``On an isoperimetric problem with a competing nonlocal
term II: The general case," Comm. Pure Appl. Math. $\mathbf{67}(12)$, $1974$-$1994$ $(2014)$.
\bibitem{KMN}  H.~Kn{\"u}pfer, C.~B. Muratov, and M.~Novaga, ``Low density phases in a
  uniformly charged liquid,'' Comm. Math. Phys. $\mathbf{345}(1)$, $141$-$183$ $(2016)$.  
\bibitem{KS}  R.V. Kohn and P. Sternberg, ``Local minimisers and singular perturbations.'' Proc. Roy. Soc. Edinburgh Sect. A  $\mathbf{111}(1$-$2)$, $69$-$84$ $(1989)$.
%
\bibitem{LeBris} C. Le Bris, ``Some results on the Thomas-Fermi-Dirac-von Weizs\"acker model," Differ. Integr. Equations $\mathbf{6}$, $337$-$353$ $(1993)$.
\bibitem{Lieb} E.H. Lieb, ``Thomas-Fermi and related theories of atoms and molecules," Rev. Mod. Phys. $\mathbf{53}(4)$, $603$-$641$ $(1981)$.
\bibitem{LiebLoss}  E. Lieb and M. Loss, \underbar{Analysis},  AMS $(2001)$.
\bibitem{LionsConcentrationPart1} P.L. Lions, ``The concentration-compactness principle in the Calculus of Variations. The locally compact case, part 1," Ann. Inst. Henri Poincar\'e, Anal. Non Lin\'eaire $\mathbf{1}$, $109$-$149$ $(1984)$.
\bibitem{Lions} P.L. Lions, ``Solutions of Hartree-Fock equations for Coulomb systems," Commun. Math. Phys. $\mathbf{109}(1)$, $33$-$97$ $(1987)$.
\bibitem{LuOttoNon-existence} J. Lu and F. Otto, ``Non-existence of a minimizer for Thomas-Fermi-Dirac-von Weizs\"acker model," Commun. Pure Appl. Math. $\mathbf{67}$, $1605$-$1617$ $(2014)$.
\bibitem{LuOttoLiquidDropBackground} J. Lu and F. Otto, ``An isoperimetric problem with Coulomb repulsion and attraction to a background nucleus,'' \textit{preprint at	arXiv:1508.07172} $(2015)$.
\bibitem{ModMort}  L. Modica and S. Mortola, ``Il limite nella $\Gamma$-convergenza di una famiglia di funzionali ellittici,'' { Boll. Un. Mat. Ital.} A $\mathbf{14}(3)$, $526$-$529$, $(1977)$.
%
 \bibitem{Muratov} C. B. Muratov, ``Droplet phases in non-local Ginzburg-Landau models with Coulomb repulsion in two dimensions,'' Comm. Math. Phys., $\mathbf{299}$, $45$-$87$ $(2010)$.
\bibitem{NamVBosch} P.T. Nam and H. Van Den Bosch, ``Non-existence in Thomas-Fermi-Dirac-von Weizs\"acker theory with small nuclear charges," Math. Phys. Anal. Geom. $\mathbf{20}(6)$ $(2017)$. 
{\bibitem{RenTruskinovsky} X.F. Ren and L. Truskinovsky, ``Finite scale microstructures in nonlocal elasticity,'' J. Elasticity, $\mathbf{59}$, $319$-$355$ $(2000)$.
\bibitem{RenWei200} X.F. Ren and J.C. Wei, ``On the multiplicity of solutions of two nonlocal variational problems,'' SIAM J. Math. Anal. $\mathbf{31}$, $909$-$924$ $(2000)$.}
\bibitem{Sternberg} P. Sternberg, ``The effect of a singular perturbation on nonconvex variational problems,'' Archive for Rational Mechanics and Analysis $\mathbf{101}(3)$, $209$-$260$ $(1988)$. 
\end{thebibliography}
\end{document}